\documentclass[onecolumn,12pt,lettersize]{IEEEtran}
\usepackage[top=1 in, bottom=1 in, left=0.75 in, right=0.75 in]{geometry}

\usepackage[cmex10]{amsmath} 
\usepackage{amssymb}
\usepackage{tabularx}
\usepackage[sc,osf,noBBpl]{mathpazo}
\usepackage{cite} 
\usepackage[pdftex]{graphicx} 
\usepackage{subcaption}
\usepackage{algorithm}
\usepackage[noend]{algorithmic}
\usepackage{amssymb} 
\usepackage{graphicx}
\usepackage{epstopdf}
\usepackage{enumerate}
\usepackage{stfloats}
\usepackage{tikz}

\usepackage{amsthm} 
\usepackage[hidelinks]{hyperref}

\newtheorem{theorem}{Theorem}
\newtheorem{lemma}{Lemma}

\newtheorem{cor}{Corollary}

\newcommand{\delete}[1]{}

\begin{document}

\title{Steady-State Analysis of Load Balancing with Coxian-$2$ Distributed Service Times}

\author{Xin Liu, Kang Gong and Lei Ying
\IEEEcompsocitemizethanks{\IEEEcompsocthanksitem Xin Liu, Kang Gong and Lei Ying are with the Electrical Engineering and Computer Science Department of
the University of Michigan, Ann Arbor, MI, 48109 USA.
Email: \{xinliuee, kanggong,  leiying\}@umich.edu}
}

\maketitle

\begin{abstract}
This paper studies load balancing for many-server ($N$ servers) systems. Each server has a buffer of size $b-1,$ and can  have at most one job in service and $b-1$ jobs in the buffer. The service time of a job follows the Coxian-2 distribution. We focus on steady-state performance of load balancing policies in the heavy traffic regime such that the normalized load of system is $\lambda = 1 - N^{-\alpha}$ for $0<\alpha<0.5.$ 
We identify a set of policies that achieve asymptotic zero waiting. The set of policies include several classical policies such as join-the-shortest-queue (JSQ), join-the-idle-queue (JIQ), idle-one-first (I1F) and power-of-$d$-choices (Po$d$) with $d=O(N^\alpha\log N)$. 
The proof of the main result is based on Stein's method and state space collapse. A key technical contribution of this paper is the iterative state space collapse approach that leads to a simple generator approximation when applying Stein's method.
\end{abstract}

\section{Introduction}
The convergence of cloud computing and machine learning is transforming society in unprecedented ways, and leading to innovations in autonomous systems, healthcare, bioinformatics, social networks, online and in-store retail industry, and education.  Data centers nowadays continuously process complex queries and machine learning tasks in large server farms, with tens of thousands of networked servers. Many of these queries/tasks are time sensitive such as queries for products on online retail platforms, realtime machine learning tasks such as language translation and virtual reality applications. In fact, the latency cost of a data center can be very high. In 2017, Akamai reported that 100-millisecond delay led to 7\% drop  in sales \cite{Aka_17}. Therefore, it is critical for a data center to process these jobs/queries in a timely fashion, ideally without any delay. This paper focuses on the following critical question: {\em can we achieve almost zero-delay in large-scale data centers?} A critical step for achieving zero-delay is a good load-balancing algorithm that can balance the load across servers and assign an incoming job to an idle server immediately.  
Assuming exponential service times, sufficient conditions under which a load balancing algorithm achieves asymptotic zero-delay have been obtain in \cite{LiuYin_20} for $0<\alpha<0.5$ and in \cite{LiuYin_19} for $0.5\leq \alpha<1.$ The results have also been extended to parallel-jobs \cite{WenWan_20}, multi-server jobs \cite{WanXieHar_21} and jobs with data locality \cite{WenZhoSri_20}. 
This paper considers jobs with Coxian-2 service times and identifies a set of load balancing algorithms that achieve zero waiting at steady-state. While Coxian-$2$ distribution is still a restricted service-time distribution, it has been widely-used in  computer systems (see, e.g. \cite{Alt_85,TelHei_03,OsoHar_03}). In particular, \cite{Alt_85} showed the Coxian-$2$ distribution can well approximate a general distribution by fitting its first three moments when the moments of the general distribution satisfies  $m_3/m_1 \geq \frac{3}{2}(c+1)^2$ and $c\geq 1$, where $m_1$, $m_3$ and $c$ are the first-order moment, third-order moment and the squared coefficient of variation, respectively.  \cite{OsoHar_03} also showed that the Coxian-$2$ distribution can represent a large class of bounded Pareto distributions, which model many real-world job service times in computing and communication systems, including UNIX I/O time and the duration of HTTP and FTP transfers.  

\subsection{Related Work}
Performance analysis of systems with distributed queues is one of the most fundamental and widely-studied problems in queueing theory. Assuming exponential service time, the steady-state performance of various load balancing policies has been analyzed using the mean-field analysis (fluid-limit analysis) or diffusion-limit analysis. Among the most popular policies are: 1) join-the-shortest-queue (JSQ) \cite{EscGam_18,Bra_20}, which routes an incoming job to the least loaded server; 2) join-the-idle-queue (JIQ) \cite{LuXieKli_11,Sto_15}, which routes an incoming job to an idle server if possible and otherwise to a server chosen uniformly at random;  3) idle-one-first (I1F) \cite{GupWal_19}, which routes an incoming job to an idle server if available and otherwise to a server with one job if available. If all servers have at least two jobs, the job is routed to a randomly selected server; and 4)
power-of-$d$-choices (Po$d$) \cite{Mit_96,VveDobKar_96}, which samples $d$ servers uniformly at random and dispatches the job to the least loaded server among the $d$ servers. 
With general service time distributions, performance analysis of load balancing policies with distributed queues is a much more challenging problem, and remains to be an active research area in queueing theory \cite{Har_13}. \cite{Mit_96} proposed a  mean-field model of the Po$d$ policy under gamma service time distributions without proving the convergence of the stochastic system to the mean-field model. \cite{AghLiRam_17, VasMukMaz_17, HelVan_18} proposed a set of PDE models to approximate load balancing polices under general service times and numerically analyzed  key performance metrics (e.g. mean response time).
They proved the convergence of the stochastic systems to the corresponding ODEs or PDEs at process-level (over a finite time interval instead of at steady state). 

To go beyond the process-level and establish steady-state performance with general service times, a key challenge is to prove that the mean-field system (fluid-system) is stable, i.e. the system converges to a unique equilibrium starting from any initial condition. Under non-exponential service time distributions, the proof of stability often relies on a so-called ``monotonicity property'', which requires a partial order of two mean-field systems starting from two initial conditions to be maintained over time. In particular,  letting $x(t, y)$ denote the system state at time $t$ with initial state $y,$ given two initial conditions $y_1 \succ y_2,$ where "$\succ$" is a certain partial order, ``monotonicity'' states that the partial order  $x(t, y_1) \succ x(t, y_2)$ holds for any $t \geq 0.$ 

Monotonicity does hold under several load balancing policies with non-exponential service time distributions that have a decreasing hazard rate (DHR) \cite{BraLuPra_12,Sto_15,FosSto_17}. The hazard rate is defined to be $\frac{f(x)}{1-F(x)},$ where $f(x)$ is the density function of the service time and $F(x)$ is the corresponding cumulative distribution function. With DHR,  \cite{BraLuPra_12} proved the asymptotic independence of queues in the mean-field limit under the  Po$d$ load balancing policy, and 
\cite{Sto_15} proved that JIQ achieves asymptotic delay optimality. \cite{Van_19} proved the global stability of the mean-filed model of load balancing policies (e.g. Po$d$) under hyper-exponential distributions with DHR. The key step in \cite{Van_19} is to represent hyper-exponential distribution by a constrained Coxian distribution, where $\mu_i(1-p_i)$ is decreasing in phase $i$ ($\mu_i$ is the service rate in phase $i$ and $p_i$ is the probability that a job finishing service in phase $i$ and entering phase $i+1$). With the alternative representation, monotonicity holds in a certain partial order and the global stability is established.

When service time distributions do not satisfy DHR,  only few works established the stability of mean-field systems for very limited light-traffic regimes.    For example, \cite{FosSto_17} relaxed DHR assumption in \cite{Sto_15} to any general service distribution but the asymptotic optimality of JIQ only holds when the normalized load $\lambda < 0.5.$
The stability of Po$d$ with any general service time distributions with finite second moment has also been established in \cite{BraLuPra_12} when the load per server the normalized load $\lambda < 1/4.$ 

The Coxian-2 distribution considered in this paper does not necessarily satisfy DHR. Under the Coxian-2 service time distribution, each job has two phases (phase 1 and phase 2). When in service, a job finishes phase 1 with rate $\mu_1;$ and after finishing phase 1, the job leaves the system with probability $1-p$ or enters phase 2 with probability $p.$ If the job enters phase 2, it finishes phase 2 with rate $\mu_2,$ and leaves the system. Consider a simple system with two servers. Assume  the Coxian-2 service time distribution and JSQ is used for load balancing. Consider two different initial conditions for this system as shown in Figure~\ref{fig:non-monoto}, where jobs in phase $1$ are in red color, jobs in phase $2$ are in green color and jobs before processed by the server are in black color.
The state of each server can be represented by its queue length and the expected remaining service time of the job in service. Let $Q^{(i,1)}(t)$ and $Q^{(i,2)}(t)$ denote the queue length of server $i$ at time $t$ with initial condition 1 and 2, respectively, and $T^{(i,1)}(t), T^{(i,2)}(t)\in\left\{\frac{1}{\mu_1}+\frac{p}{\mu_2}, \frac{1}{\mu_2}, 0\right\}$ denote the expected remaining service time of the job in service at server $i$ with initial condition 1 and 2, respectively.
At time $0,$ we have $Q^{(i,1)}(0) \geq Q^{(i,2)}(0)$ and $T^{(i,1)}(0) \geq T^{(i,2)}(0)$ for all $i=1,2.$ During the time period $(0, t_1],$ two jobs arrived and were routed to servers according to JSQ, which resulted in the state shown in Figure~\ref{fig:non-monoto}. Suppose that $(1-p)\mu_1<\mu_2,$ then  at time $t_1,$ we have  $T^{(2,1)}(t_1)=\frac{1}{\mu_2}<T^{(2,2)}(t_1)=\frac{1}{\mu_1}+\frac{p}{\mu_2},$ so the system does not have mononticity. Note the hazard rate of Coxian-2 distribution is $\frac{f(x)}{1-F(x)} = \frac{(1-p)\mu_1 + \mu_2e^{(1+p)\mu_1 x}}{1 + e^{(1+p)\mu_1 x}},$ which is an increasing function for $(1-p)\mu_1<\mu_2,$ therefore, it does not satisfy the DHR property.
\begin{figure}[!htbp]
\begin{center}
\includegraphics[width=3.1in]{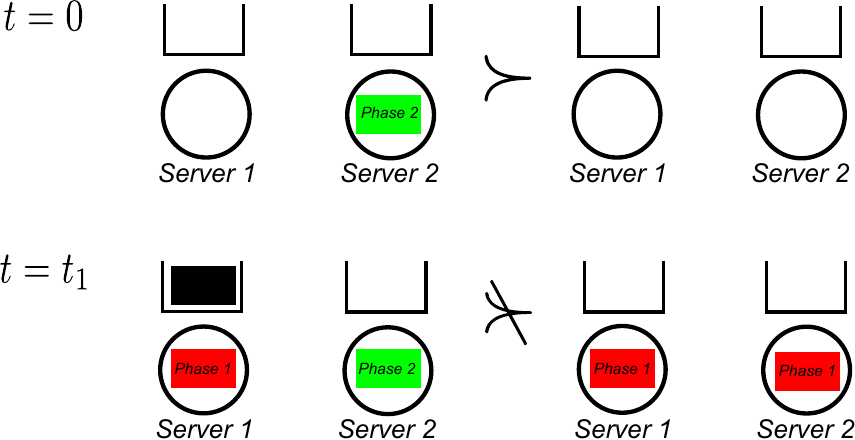}
\end{center}
\caption{~Non-monotocity of JSQ under Coxian-2 distribution.\label{fig:non-monoto}}
\end{figure}

\subsection{Main Contributions}
In this paper, we analyze the steady-state performance of many server systems assuming Coxian service time distributions and heavy traffic regimes ($\lambda = 1 - N^{-\alpha}$ for $0 < \alpha < 0.5$).   From the best of our knowledge, this is the first paper that establishes the steady-state performance of general Coxian distributions without DHR in heavy-traffic regimes. In this paper,  we develop an iterative state space collapse (SSC) to show the steady-state ``lives" in a restricted region (with a high probability), in which the original system is coupled with a simple system by Stein's method. With iterative SSC and Stein's method, we are able to establish several key performance metrics at steady state, including the expected queue length, the probability that a job is allocated to a busy server (waiting probability) and the waiting time. The main results include: 
\begin{itemize}
\item For any load balancing policy in a policy set $\Pi$ (the detailed definition is given in \eqref{Piset}), which  includes join-the-shortest-queue (JSQ), join-the-idle-queue (JIQ), idle-one-first (I1F) and power-of-$d$-choices (Po$d$) with $d=O(N^\alpha\log N)$ , the mean queue length is $\lambda + O\left(\frac{\log N}{\sqrt{N}}\right).$ 
\item For JSQ and Po$d$ with $d=O(N^\alpha\log N),$ the waiting probability and the expected waiting time per job are both $O\left(\frac{\log N}{\sqrt{N}}\right).$ 
\item For JIQ and I1F, the waiting probability is $O\left(\frac{1}{N^{0.5-\alpha}\log N}\right).$
\end{itemize}

\section{Model and Main Results}
We consider a many-server system with $N$ homogeneous servers, where job arrival follows a Poisson process with rate $\lambda N$ with $\lambda = 1 - N^{-\alpha}, 0<\alpha <0.5$ and service times follow Coxian-2 distribution  ($\mu_1$, $\mu_2$, $p$)  as shown in Figure~\ref{fig:cox2}, where $\mu_m>0$ is the rate a job finishes phase $m$ when in service and $0\leq p<1$ is the probability that a job enters phase 2 after finishing  phase 1. We assume $\lambda = 1 - N^{-\alpha}$ for ease of exposition. Our results hold for $\lambda = 1 - \beta N^{-\alpha}$ with any constant $\beta >0,$ and the extension is straightforward.
\begin{figure}[!htbp]
  \centering
  \includegraphics[width=1.9in]{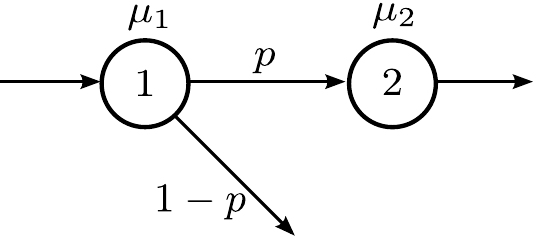}
  \caption{~Coxian-2 distribution.}
  \label{fig:cox2}
\end{figure}

Without loss of generality, we assume the mean service time to be one, i.e. $$\frac{1}{\mu_1}+\frac{p}{\mu_2}=1.$$ 
Under this assumption, $\lambda$ is also the load of the system.

As shown in Figure~\ref{model-coxian2}, an arrival job is colored with black before processed by the server, and colored with red and green when it is in phase $1$ and phase $2$ in service, respectively.
Each server has a buffer of size $b-1,$ so can hold at most $b$ jobs ($b-1$ in the buffer and one in service). 

\begin{figure}[!htbp]
  \centering
  \includegraphics[width=2.5in]{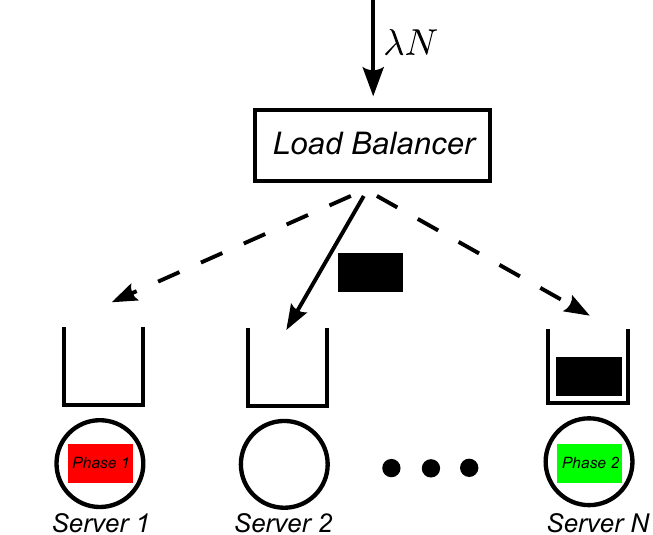}
  \caption{~Load Balancing in Many-Server Systems.}
  \label{model-coxian2}
\end{figure}

Let $Q_{j,m}(t)$ ($m=1,2$) denote the fraction of servers which have $j$ jobs at time $t$ and the one in service is in phase $m$. For convenience, we define $Q_{0,1}(t)$ to be the fraction of servers that are idle at time $t$ and $Q_{0,2}(t)=0.$ Furthermore, define $Q(t)$ to be a $b\times 2$ matrix such that the $(j,m)$th entry of the matrix is $Q_{j,m}(t).$ 
Define $S_{i,m}(t) = \sum_{j \geq i} Q_{j,m}(t)$ and $S_{i}(t) = \sum_{m=1}^2 S_{i,m}(t).$ In other words,  $S_{i,m}(t)$ is the fraction of servers which have at least $i$ jobs and the job in service is in phase $m$ at time $t$ and $S_i(t)$ is the fraction of servers with at least $i$ jobs at time $t.$ Furthermore define $S(t)$ to be a $b\times 2$ matrix such that the $(j,m)$th entry of the matrix is $S_{j,m}(t).$ Note $Q(t)$ and $S(t)$ have an one-to-one mapping.
We consider load balancing policies which dispatch jobs to servers based on $Q(t)$ (or $S(t)$) and under which the finite-state CTMC $\{Q(t), t\geq 0\}$ (or $\{S(t), t\geq 0\}$) is irreducible, and so it has a unique stationary distribution. The load balancing policies include JSQ, JIQ, I1F and Po$d$.

Let $Q_{j,m}$ denote $Q_{j,m}(t)$ at steady state. We further define $S_{i,m}= \sum_{j \geq i} Q_{j,m}$ and $S_{i} = \sum_{m} S_{i,m}.$ In other words,  $S_{i,m}$ is the fraction of servers which have at least $i$ jobs and the job in service is in phase $m$ and $S_i$ is the fraction of servers with at least $i$ jobs at steady state. We illustrate the state representation $S_{i,m}$ in Figure~\ref{model-state} and Table \ref{tab:exa}.
\begin{figure}[!htbp]
  \centering
  \includegraphics[width=5.6in]{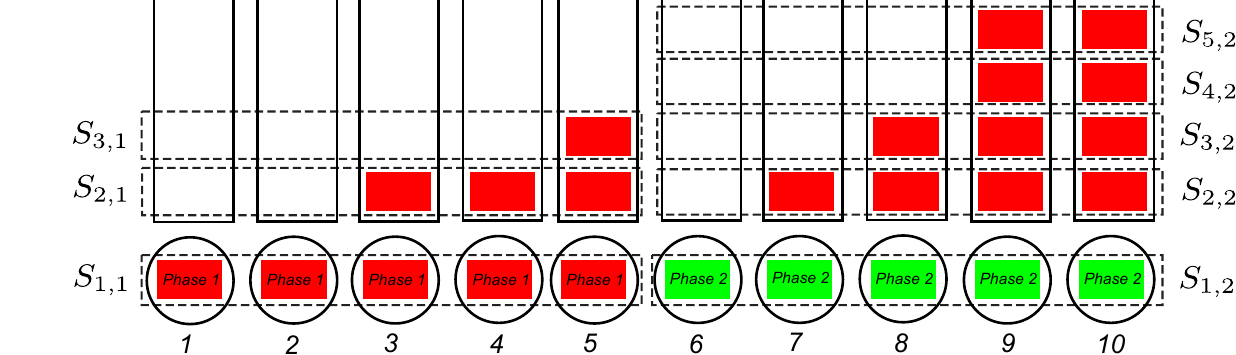}
  \caption{~Illustrations of states $S_{i,m}.$}
  \label{model-state}
\end{figure}

\begin{table}[!htbp]
\centering
\begin{tabular}{lll|lllll}
$Q_{1,1}$               & $Q_{2,1}$               & $Q_{3,1}$                & $Q_{1,2}$               & $Q_{2,2}$               & $Q_{3,2}$               & $Q_{4,2}$               & $Q_{5,2}$               \\ \hline
\multicolumn{1}{c}{0.2} & \multicolumn{1}{c}{0.2} & \multicolumn{1}{c|}{0.1} & \multicolumn{1}{c}{0.1}   & \multicolumn{1}{c}{0.1} & \multicolumn{1}{c}{0.1} & \multicolumn{1}{c}{0}   & \multicolumn{1}{c}{0.2} \\ \hline
$S_{1,1}$               & $S_{2,1}$               & $S_{3,1}$                & $S_{1,2}$               & $S_{2,2}$               & $S_{3,2}$               & $S_{4,2}$               & $S_{5,2}$               \\ \hline
\multicolumn{1}{c}{0.5} & \multicolumn{1}{c}{0.3} & \multicolumn{1}{c|}{0.1} & \multicolumn{1}{c}{0.5} & \multicolumn{1}{c}{0.4} & \multicolumn{1}{c}{0.3} & \multicolumn{1}{c}{0.2} & \multicolumn{1}{c}{0.2}
\end{tabular}
\caption{~Values of $Q_{i,m}$ and $S_{i,m}$ in Figure~\ref{model-state}. \label{tab:exa}}
\end{table}

Define $S$ to be a $b\times 2$ random matrix such that the $(i,m)$th entry is  $S_{i,m}$ and let $s \in \mathbb R^{b\times 2}$ denote a realization of $S.$ Define $\mathcal S^{(N)}$ to be a set of $s$ such that  \begin{align}
\mathcal{S}^{(N)}  = \left\{ s ~\left|~ 1\geq s_{1,m}\geq \cdots \geq s_{b,m}\geq  0, ~ 1\geq \sum_{m=1}^2s_{1,m}; ~ Ns_{i,m}\in \mathbb{N}, ~ \forall i, m\right.\right\} \label{state_S}.
\end{align}

Let $A_1(s)$ denote the probability that an incoming job is routed to a busy server conditioned on that the system is in state $s \in \mathcal{S}^{(N)};$ i.e.  $$A_1(s)= \mathbb P \left(\left. \text{an incoming job is routed to a busy server} \right| {S(t)=s}\right).$$  Among the load balancing policies considered in this paper, define a subset \begin{align}\Pi=\left\{\pi~ \left|~ \text{Under policy}~ \pi, A_1(s)\leq \frac{1}{\sqrt{N}} ~\text{for any}~ s\in{\mathcal{S}^{(N)}} \right.\right. \nonumber\\ \left.\left.~\text{such that}~s_1\leq \lambda+\frac{1+\mu_1+\mu_2}{\min\left\{(1-p)\mu_1,\mu_2\right\}}\frac{\log N}{\sqrt{N}}\right.\right\}. \label{Piset}
\end{align}Our main result of this paper is the following theorem.

\begin{theorem}\label{Thm:main}
Define $w_u=\max\{(1-p)\mu_1,\mu_2\},$ $w_l=\min\{(1-p)\mu_1,\mu_2\},$ $\mu_{\max}=\max\{\mu_1,\mu_2\},$ and $k =\left(1+\frac{w_u b}{w_l}\right)\left(\frac{1+\mu_1+\mu_2}{w_l}+2\mu_1\right).$ Under any load balancing policy in $\Pi,$ the following bound holds 
\begin{align}
\mathbb E\left[\max\left\{\sum_{i=1}^{b} S_i- \lambda -\frac{k\log N}{\sqrt{N}},0\right\}\right] \leq \frac{7\mu_{\max}}{\sqrt{N}\log N}, \label{Thm:main-S}
\end{align}
when $N$ satisfies
\begin{align}
\frac{w_lN^{0.5-\alpha}}{1+\mu_1+\mu_2} \geq \log N \geq \frac{3.5}{\min\left(\frac{\mu_1}{16\mu_{\max}},\frac{\mu_2}{12\mu_{\max}},\frac{\mu_1\mu_2}{40\mu_{\max}}\right)}. \label{N-cond}    
\end{align}
\hfill{$\square$}
\end{theorem}

Note that the condition $A_1(s)\leq \frac{1}{\sqrt{N}}$ for $s$ such that $s_1\leq \lambda+\frac{1+\mu_1+\mu_2}{w_l}\frac{\log N}{\sqrt{N}}$ means that an incoming job is routed to an idle server with probability at least $1-\frac{1}{\sqrt{N}}$ when at least $\frac{1}{N^{\alpha}}-\frac{1+\mu_1+\mu_2}{w_l}\frac{\log N}{\sqrt{N}}$ fraction of servers are idle. 
There are several well-known policies that satisfy this condition.
\begin{itemize}
\item {\bf Join-the-Shortest-Queue (JSQ)}: JSQ routes an incoming job to the least loaded server in the system. Therefore, $A_1(s)=0$ when $s_{1}<1.$

\item {\bf Idle-One-First (I1F)} \cite{GupWal_19}: I1F routes an incoming job to an idle server if available; and otherwise to a server with one job if available. If all servers have at least two jobs, the job is routed to a randomly selected server. Therefore, $A_1(s)=0$ when $s_1<1.$

\item {\bf Join-the-Idle-Queue (JIQ)} \cite{LuXieKli_11}: JIQ routes an incoming job to an idle server if possible and otherwise, routes to a server chosen uniformly at random.
Therefore, $A_1(s)=0$ when $s_1<1.$

\item {\bf Power-of-$d$-Choices (Po$d$)} \cite{Mit_96,VveDobKar_96}:  Po$d$ samples $d$ servers uniformly at random and dispatches the job to the least loaded server among the $d$ servers. Ties are broken uniformly at random. When $d\geq \mu_1N^\alpha\log N,$ $A_1(s)\leq \frac{1}{\sqrt{N}}$ when $s_1\leq \lambda+\frac{1+\mu_1+\mu_2}{w_l}\frac{\log N}{\sqrt{N}}.$
\end{itemize}

A direct consequence of Theorem \ref{Thm:main} is {\it{asymptotic zero waiting} at steady state}.  Let $\mathcal W$ denote the event that an incoming job is routed to a busy server in a system with $N$ servers, and $\mathbb P(\mathcal W)$ denote the probability of this event at steady-state. Let $\mathcal B$ denote the event that an incoming job is blocked (discarded) and $\mathbb P(\mathcal B)$ denote the probability of this event at steady-state. Note that the occurrence of event ${\mathcal B}$ implies the occurrence of event ${\mathcal W}$ because a job is blocked when being routed to a server with $b$ jobs. 
Furthermore, let $W$ denote the waiting time of a job (when the job is not dropped). We have the following results based on the main theorem.

\begin{cor} \label{Thm:zerodelay}
The following results hold when $N$ satisfies condition \eqref{N-cond}.
\begin{itemize}
\item Under JSQ and Po$d$ with $d\geq \mu_1 N^\alpha \log N$ such that $\sqrt{N} \geq \frac{8k\log N}{b-\lambda}+\frac{8b N^{0.5-\alpha}}{(b-\lambda)\mu_1},$ we have
\begin{align} 
\mathbb E\left[W\right]\leq& \frac{2k\log N}{\sqrt{N}} + \frac{14\mu_{\max} + \frac{16\mu_{\max}}{b-\lambda}}{\sqrt{N} \log N}, \label{Pod-zero-delay}\\
\mathbb P(\mathcal W) \leq& \frac{1}{N}+ \frac{\mu_{\max}}{\lambda }\left(\frac{k\log N}{\sqrt{N}} +  \frac{7\mu_{\max} + \frac{8 \mu_{\max}}{b-\lambda}}{\sqrt{N} \log N}\right)\label{Pod-zero-prob}.
\end{align}

\item Under JIQ and I1F such that $N^{0.5-\alpha} \geq 2k\log N,$ \begin{align}\mathbb P(\mathcal W)\leq \frac{14\mu_{\max}}{N^{0.5-\alpha}\log N}.\label{JIQ-zero-waiting}\end{align}
\end{itemize}
\hfill{$\square$}
\end{cor}

The proof of this corollary is an application of  Little's law and  Markov's inequality, and can be found in the Section \ref{sec:0-delay}. We  remark that according to  \eqref{Pod-zero-delay} and \eqref{Pod-zero-prob}, asymptotic zero-waiting is achieved under JSQ and Po$d$ when $k=o\left(\frac{\sqrt{N}}{\log N}\right);$  and according to \eqref{JIQ-zero-waiting}, 
 asymptotic zero-waiting is achieved  under JIQ and I1F
when $k=O\left(\frac{N^{0.5-\alpha}}{\log N}\right).$ Since Theorem \ref{Thm:main} assumes  $k=\Theta(b),$ the buffer size has to be $O\left(\frac{N^{0.5-\alpha}}{\log N}\right)$ as well, which results in the finite-buffer assumption in this paper. This finite-buffer assumption, however, is a sufficient condition. It remains open whether such a condition is necessary.

\def\Thmmain{\ref{Thm:main}~}

\section{Proof of Theorem \Thmmain}   \label{sec:proof}
In this section, we present the proof of our main theorem, which is organized along the three key ingredients: generator approximation, gradient bounds, and iterative state space collapse.

\subsection{Generator Approximation}

Define $e_{i,m} \in \mathbb R^{b\times 2}$ to be a $b\times 2$-dimensional matrix such that the $(i,m)$th entry is $1/N$ and all other entries are zero. Furthermore, define $A_{i,m}(s)$ to be the probability that an incoming job is routed to a server with at least $i$ jobs and the job in service in phase $m$, when the system is in state $s$, i.e.
\begin{align*}
A_{i,m}(s)=\Pr\left( \hbox{an incoming job is routed to a server with at least $i$ jobs} \right.\\
\left. \hbox{and the job in service in phase $m$} ~|~{S(t)}=s\right).    
\end{align*}

Given the state $s$ of the CTMC and the corresponding $q,$ the following events trigger a
transition from state $s$.
\begin{itemize}
\item Event 1: A job arrives and is routed to a server that it has $i-1$ jobs and the job in service is in phase $1.$ When this occurs, $q_{i,1}$ increases by $1/N,$ and $q_{i-1,1}$ decreases by $1/N,$ so the CTMC has the following transition: 
\begin{align*}
q \to& ~q+e_{i,1}-e_{i-1,1},\\
s \to& ~s+e_{i,1}.
\end{align*}

This transition occurs with rate $$\lambda N(A_{i-1,1}(s)-A_{i,1}(s)),$$ where $A_{i-1,1}(s)-A_{i,1}(s))$ is the probability that an incoming job is routing to a server with $i-1$ jobs and the job in service in phase $1.$ For example, under JSQ, we have $A_{i-1,1}(s)-A_{i,1}(s))=\frac{q_{i-1,1}}{q_{i-1}}{\mathbb{I}}_{\{s_{i-1}=1, s_i<1\}},$ where $\frac{q_{i-1,1}}{q_{i-1}}$ is the probability that the server which receives the job is serving a job in phase $1$ conditioned on the job is routed to a server with $i-1$ jobs, and $\{s_{i-1}=1, s_i<1\}$ implies that the shortest queue in the system has length $i-1.$

\item Event 2: A job arrives and is routed to a server such that it has $i-1$ jobs and the job in service is in phase $2.$  When this occurs, $q_{i,2}$ increases by $1/N,$ and $q_{i-1,2}$ decreases by $1/N,$ so the CTMC has the following transition: 
\begin{align*}
q \to& ~q+e_{i,2}-e_{i-1,2},\\
s \to& ~s+e_{i,2}.
\end{align*}

This transition occurs with rate $$\lambda N(A_{i-1,2}(s)-A_{i,2}(s)),$$ where $A_{i-1,2}(s)-A_{i,2}(s))$ is the probability that an incoming job is routing to a server with $i-1$ jobs and the job in service in phase $2.$ For example, under JSQ, we have $A_{i-1,2}(s)-A_{i,2}(s))=\frac{q_{i-1,2}}{q_{i-1}}{\mathbb{I}}_{\{s_{i-1}=1, s_i<1\}},$ where $\frac{q_{i-1,2}}{q_{i-1}}$ is the probability that the server which receives the job is serving a job in phase $2$ conditioned on the job is routed to a server with $i-1$ jobs, and $\{s_{i-1}=1, s_i<1\}$ implies that the shortest queue in the system has length $i-1.$

\item Event 3: A server, which has $i$ jobs, finishes  phase 1 of the job in service. The job leaves the system without entering into phase 2. When this occurs, $q_{i,1}$ decreases by $1/N$ and $q_{i-1,1}$ increases by $1/N,$ so the CTMC has the following transition:
\begin{align*}
q \to& ~q-e_{i,1}+e_{i-1,1},\\
s \to& ~s-e_{i,1}.
\end{align*}

This transition occurs with rate $$\mu_1Nq_{i,1}(1-p),$$
where $(1-p)$ is the probability that a job finishes phase 1 and departures without entering phase 2. 

\item Event 4: A server, which has $i$ jobs, finishes phase 1 of the job in service. The job enters phase $2$. 
When this occurs, a server in state $(i,1)$ transits to state $(i,2),$ so $q_{i,1}$ decreases by $1/N$ and $q_{i,2}$ increases by $1/N.$ Therefore, the CTMC has the following transition:
\begin{align*}
q \to& ~q-e_{i,1}+e_{i,2},\\
s \to& ~s-\sum_{j=1}^i e_{j,1} + \sum_{j=1}^i e_{j,2},
\end{align*} where the transition of $s$ can be verified based on the definition $s_{i,m}=\sum_{j\geq i} q_{j,m}$ so $s_{j,1}$ decreases by $1/N$ for any $j\leq i$ and $s_{j,2}$ increases by $1/N$ for any $j\leq i.$ This event occurs with rate $$\mu_1Nq_{i,1}p,$$
where $p$ is the probability that a job enters phase 2 after finishing  phase $1.$ 

\item Event 5: A server, which has $i$ jobs, finishes phase 2 of the job in service. The job leaves the system. When this occurs, $q_{i,2}$ decreases by $1/N$ and $q_{i-1,1}$ increases by $1/N$ (because the server starts a new job in phase 1 {and the event when $i=1$ means the fraction of idle server increase by $1/N$}), so the CTMC has the following transition:
\begin{align*}
q \to& ~q-e_{i,2}+e_{i-1,1},\\
s \to& ~s-\sum_{j=1}^i e_{j,2} + \sum_{j=1}^{i-1} e_{j,1}.
\end{align*}

This transition occurs with rate $$\mu_2Nq_{i,2}.$$
\end{itemize}

We illustrate local state transitions related to state $s$ under JSQ in Figure~ \ref{fig:state-trans}.
\begin{figure}[!htbp]
  \centering
\includegraphics[width=4.7in]{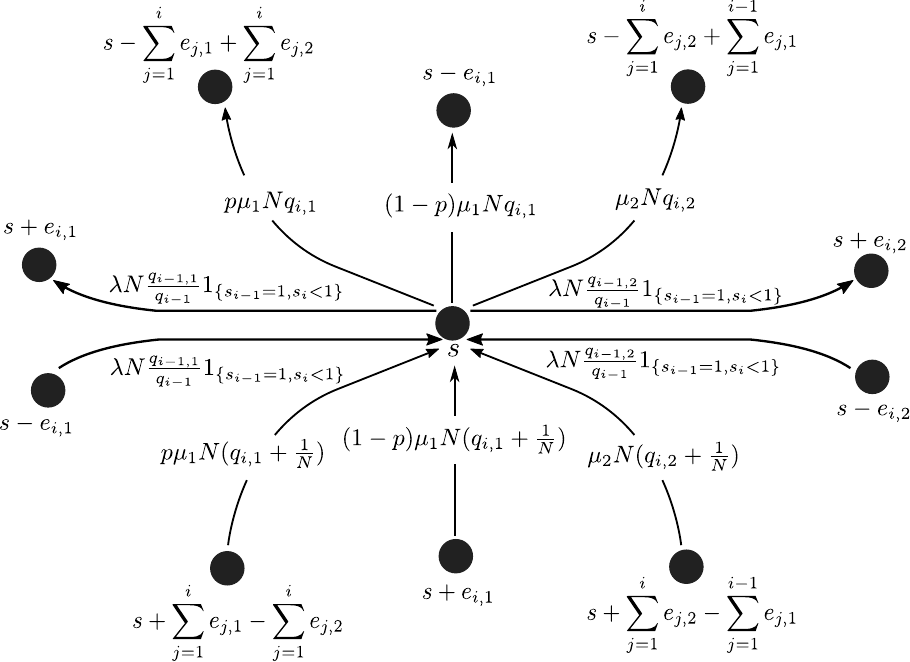} 
  \caption{~Illustrations of state transitions under JSQ for any $i$ with $1 \leq i \leq b.$}
  \label{fig:state-trans}
\end{figure}

Let $G$ be the generator of CTMC ($S(t):t\geq 0$). Given function $f: \mathcal S^{(N)} \to \mathbb{R},$ we have
\begin{align}
G f(s) = &\sum_{i=1}^{b}\left[\lambda N(A_{i-1,1}(s)-A_{i,1}(s))(f(s + e_{i,1}) - f(s))  \right.  \label{G-Si1-A}\\ 
&~~~~~\left. + \lambda N(A_{i-1,2}(s)-A_{i,2}(s))(f(s + e_{i,2}) - f(s))  \right.  \label{G-Si2-A}\\ 
&~~~~~ + (1-p)\mu_1 Nq_{i,1}(f(s - e_{i,1})-f(s))  \label{G-Si1-D}\\
&~~~~~ + p\mu_1 Nq_{i,1}\left(f\left(s - \sum_{j=1}^ie_{j,1}+\sum_{j=1}^ie_{j,2}\right)-f(s)\right)  \label{G-Si1-2}\\
&~~~~~\left. + \mu_2 Nq_{i,2}\left(f\left(s -  \sum_{j=1}^i e_{j,2} + \sum_{j=1}^{i-1} e_{j,1}\right)-f(s)\right)\right].  
\label{Gen:G}
\end{align}  
For any bounded function $f: \mathcal S^{(N)} \to \mathbb{R},$
\begin{align}
\mathbb E[ G f(S) ] = 0, \label{steady-cond}
\end{align}
{which can be easily verified by using the global balance equations and the fact that $S$ represents the steady-state of the CTMC.}

To understand the steady-state performance of a load balancing policy, we will establish an upper bound on the distance function in \eqref{Thm:main-S}:
$$\max\left\{\sum_{i=1}^bS_i-\eta,0\right\},$$ with 
\begin{align}
\eta = \lambda +\frac{k\log N}{\sqrt{N}}.    
\end{align}
The upper bound measures the quantity that the total number of jobs in the system ($N\sum_{i=1}^b S_i$) exceeds $N\lambda+k\sqrt{N}\log N$ at steady state, and can be used to bound the probability that an incoming job is routed to an idle server in Corollary \ref{Thm:zerodelay}.

We consider a simple fluid system with arrival rate $\lambda$ and departure rate $\lambda+\frac{\log N}{\sqrt{N}},$ i.e.
$$\dot{x}=-\frac{\log N}{\sqrt{N}},$$ and function $g(x)$ which is the solution of the following Stein's equation \cite{Yin_16}:
\begin{align}
g'(x) \left(-\frac{\log N}{\sqrt{N}}\right)  = \max\left\{x-\eta,0\right\},  \forall x, \label{Gen:L}
\end{align} where {$g'(x) = \frac{d g(x)}{d x}$}. The left-hand side of \eqref{Gen:L} can be viewed as applying the generator of the simple fluid system to function $g(x),$ i.e. $$\frac{dg(x)}{dt}=g'(x)\dot{x}=g'(x) \left(-\frac{\log N}{\sqrt{N}}\right).$$

It is easy to verify that the solution to \eqref{Gen:L} is
\begin{equation}g(x)=-\frac{\sqrt{N}}{2\log N}\left(x-\eta\right)^{2} \mathbb{I}_{x\geq \eta},\label{eq:L}\end{equation} and
\begin{equation}g'(x)=-\frac{\sqrt{N}}{\log N}\left(x-\eta\right) \mathbb{I}_{x\geq \eta}.\label{eq:Lder}\end{equation}

We note that the simple fluid system is a one-dimensional system and the stochastic system is $b\times 2$-dimensional. In order to couple these two systems, we define
\begin{equation}
f(s)=g\left(\sum_{i=1}^b\sum_{m=1}^2 s_{i,m}\right),\label{eq:steinsolution}
\end{equation} and invoke $f(s)$ in Stein's method.

Since $\sum_{i=1}^b\sum_{m=1}^2 s_{i,m} = \sum_{i=1}^b s_i\leq b$ for $s\in \mathcal{S}^{(N)},$ and $f(s)$ is bounded for $s\in\mathcal{S}^{(N)},$ we have
\begin{equation}
\mathbb E[Gf(S)]=\mathbb E\left[Gg\left(\sum_{i=1}^b\sum_{m=1}^2 S_{i,m}\right)\right]=0.\label{eq:sse}
        \end{equation}

Now define $$h(x)=\max\left\{x-\eta,0\right\}.$$ Based on \eqref{Gen:L} and \eqref{eq:sse}, we obtain
\begin{align}
&\mathbb E\left[h\left(\sum_{i=1}^b\sum_{m=1}^2 S_{i,m}\right)\right] \nonumber\\
= & \mathbb E\left[ g'\left(\sum_{i=1}^b\sum_{m=1}^2 S_{i,m}\right) \left(-\frac{\log N}{\sqrt{N}}\right) - Gg\left(\sum_{i=1}^b\sum_{m=1}^2 S_{i,m}\right)\right].\label{eq:gencou}
\end{align}

Note that according to the definition of $f(s)$ in \eqref{eq:steinsolution}, $e_{j,1}$ and $e_{j,2}$, we have $$f(s+e_{j,1})=g\left(\sum_{i=1}^b\sum_{m=1}^2 s_{i,m}+\frac{1}{N}\right), ~ f(s+e_{j,2})=g\left(\sum_{i=1}^b\sum_{m=1}^2 s_{i,m}+\frac{1}{N}\right)$$ and
$$f(s-e_{j,1})=g\left(\sum_{i=1}^b\sum_{m=1}^2 s_{i,m}-\frac{1}{N}\right), ~ f(s-e_{j,2})=g\left(\sum_{i=1}^b\sum_{m=1}^2 s_{i,m}-\frac{1}{N}\right)$$for any $1\leq j \leq b.$ Therefore,
\begin{align*}
    &Gg\left(\sum_{i=1}^b\sum_{m=1}^2 s_{i,m}\right) \\
    =&
    N \lambda\left(1-A_b(S)\right)\left(g\left(\sum_{i=1}^b\sum_{m=1}^2 s_{i,m}+\frac{1}{N}\right)-g\left(\sum_{i=1}^b\sum_{m=1}^2 s_{i,m}\right)\right)\\
    &+N \left((1-p)\mu_1 s_{1,1}+\mu_2s_{1,2}\right)\left(g\left(\sum_{i=1}^b\sum_{m=1}^2 s_{i,m}-\frac{1}{N}\right)-g\left(\sum_{i=1}^b\sum_{m=1}^2 s_{i,m}\right)\right),
\end{align*}
where the first term represents the transitions when a job arrives and the second term represents the transitions when a job departures from the system.
Note $(1-p)\mu_1s_{1,1}$ and $\mu_2s_{1,2}$ are the rates at which jobs leave the system when in phase 1 and phase 2, respectively in the state $s$. Therefore, $(1-p)\mu_1s_{1,1}+\mu_2s_{1,2}$ is the total departure rate. 
Define $d_1=(1-p)\mu_1 s_{1,1}+\mu_2s_{1,2}$ and its stochastic correspondence $D_1=(1-p)\mu_1 S_{1,1}+\mu_2S_{1,2}$ for simple notations.  

Substituting the equation above to \eqref{eq:gencou}, we have \begin{align}
&\mathbb E\left[h\left(\sum_{i=1}^b\sum_{m=1}^2 S_{i,m}\right)\right]\nonumber\\
=&\mathbb E\left[g'\left(\sum_{i=1}^b\sum_{m=1}^2 S_{i,m}\right)\left(-\frac{\log N}{\sqrt{N}}\right) \right. \nonumber\\
&\left. -N\lambda(1-A_b(S))\left(g\left(\sum_{i=1}^b\sum_{m=1}^2 S_{i,m}+\frac{1}{N}\right)-g\left(\sum_{i=1}^b\sum_{m=1}^2 S_{i,m}\right)\right)\nonumber\right.\\
&\left.-N D_1\left(g\left(\sum_{i=1}^b\sum_{m=1}^2 S_{i,m}-\frac{1}{N}\right)-g\left(\sum_{i=1}^b\sum_{m=1}^2 S_{i,m}\right)\right)\right]. \label{gen-diff-expand}
\end{align}

From the closed-forms of $g$ and $g'$ in \eqref{eq:L} and \eqref{eq:Lder}, note that for any $x <  \eta,$
\begin{equation*}
g(x) = g'\left(x\right)=0.
\end{equation*} Also note that when $x>\eta+\frac{1}{N},$
\begin{equation}g'(x)=-\frac{\sqrt{N}}{\log N}\left(x-\eta\right),\end{equation} so for $x>\eta+\frac{1}{N},$
\begin{equation}g''(x)=-\frac{\sqrt{N}}{\log N}.\end{equation}

By using mean-value theorem in the region $\mathcal T_1 = \{x ~|~ \eta-\frac{1}{N} \leq x \leq \eta+\frac{1}{N}\}$ and Taylor theorem in the region $\mathcal T_2 = \{ x ~|~ x > \eta+\frac{1}{N}\},$ we have
\begin{align}
g(x+\frac{1}{N})-g\left(x\right) =& \left(g(x+\frac{1}{N})-g\left(x\right)\right) \left(\mathbb{I}_{x \in \mathcal T_1} + \mathbb{I}_{x \in \mathcal T_2} \right) \nonumber\\
=&  \frac{g'(\xi)}{N}\mathbb{I}_{x \in \mathcal T_1} + \left(\frac{g'(x)}{N} +  \frac{g''(\zeta)}{2N^2}\right) \mathbb{I}_{x \in \mathcal T_2} \label{g+1/N}\\
g(x-\frac{1}{N})-g\left(x\right) =& \left(g(x-\frac{1}{N})-g\left(x\right)\right) \left(\mathbb{I}_{x \in \mathcal T_1} + \mathbb{I}_{x \in \mathcal T_2} \right) \nonumber\\
=&  -\frac{g'(\tilde{\xi})}{N}\mathbb{I}_{x \in \mathcal T_1} + \left(-\frac{g'(x)}{N} +  \frac{g''(\tilde{\zeta})}{2N^2}\right) \mathbb{I}_{x \in \mathcal T_2} \label{g-1/N}
\end{align}
where $\xi, \zeta \in (x,x+\frac{1}{N})$ and $\tilde{\xi}, \tilde{\zeta} \in (x-\frac{1}{N},x).$
Substitute \eqref{g+1/N} and \eqref{g-1/N} into the generator difference in \eqref{gen-diff-expand}, we have
\begin{align}
&\mathbb E\left[h\left(\sum_{i=1}^b  S_{i}\right)\right]
= J_1 + J_2 + J_3,
\end{align}
with
\begin{align}
J_1 =& \mathbb E\left[g'\left(\sum_{i=1}^b S_{i}\right)\left(\lambda A_b(S) - \lambda-\frac{\log N}{\sqrt{N}}+D_1\right)\mathbb{I}_{\sum_{i=1}^b S_{i} \in \mathcal T_2}\right], \label{G-expansion-SSC}\\
J_2=&\mathbb E\left[\left(g'\left(\sum_{i=1}^b S_{i}\right)\left(-\frac{\log N}{\sqrt{N}}\right)-\lambda(1-A_b(S))g'(\xi)+D_1g'(\tilde{\xi}) \right)\mathbb{I}_{\sum_{i=1}^b S_{i} \in \mathcal T_1} \right], \label{G-expansion-Gradient-1}\\
J_3=&-\mathbb E\left[\frac{1}{2N}\left(\lambda(1-A_b(S))g''(\zeta) + D_1g''(\tilde{\zeta})\right)\mathbb{I}_{\sum_{i=1}^b S_{i} \in \mathcal T_2}\right]. \label{G-expansion-Gradient-2}
\end{align}
Note that in \eqref{G-expansion-Gradient-1} and \eqref{G-expansion-Gradient-2}, we have that
$$\xi,\zeta  \in \left(\sum_{i=1}^b S_i,\sum_{i=1}^b S_i+\frac{1}{N}\right) ~\text{and}~ \tilde{\xi},\tilde{\zeta}\in \left(\sum_{i=1}^b S_i-\frac{1}{N},\sum_{i=1}^b S_i\right)$$ are random variables whose values depend on $\sum_{i=1}^b S_i.$ We do not include $\sum_{i=1}^b S_i$ in the notation for simplicity.

To establish the main result in Theorem \ref{Thm:main}, we need to provide the upper bounds on \eqref{G-expansion-SSC}, \eqref{G-expansion-Gradient-1} and \eqref{G-expansion-Gradient-2}. In the following subsection \ref{sec:Grad_Bounds}, we study $g'$ and $g''$ to bound the terms in \eqref{G-expansion-Gradient-1} and \eqref{G-expansion-Gradient-2}; In the subsection \ref{sec:SSC}, we study SSC to bound the term in \eqref{G-expansion-SSC}. We summarize the proof in a roadmap in Figure \ref{fig:roadmap}. Lemmas \ref{lemma:g'} and \ref{lemma:g''} establish gradient bounds, which are used to bound $J_2+J_3$ in Lemma \ref{Gradient:final}. Lemmas \ref{tailbound-u:s12}, \ref{tailbound-l:s11}, \ref{tailbound-l:s12} and \ref{SSC:s1 s2} are iterative SSC to show the system is in $S_{ssc}$ with a high probability, which rely on Lemma \ref{tail-bound-cond} and are used to bound $J_1$ in Lemmas \ref{SSC:s1 s2 negative} and \ref{SSC:out}. We finally prove Theorem \ref{Thm:main} by combining Lemmas \ref{Gradient:final}, \ref{SSC:s1 s2 negative} and \ref{SSC:out}.
\begin{figure}[!htbp]
  \centering
  \includegraphics[width=4.1in]{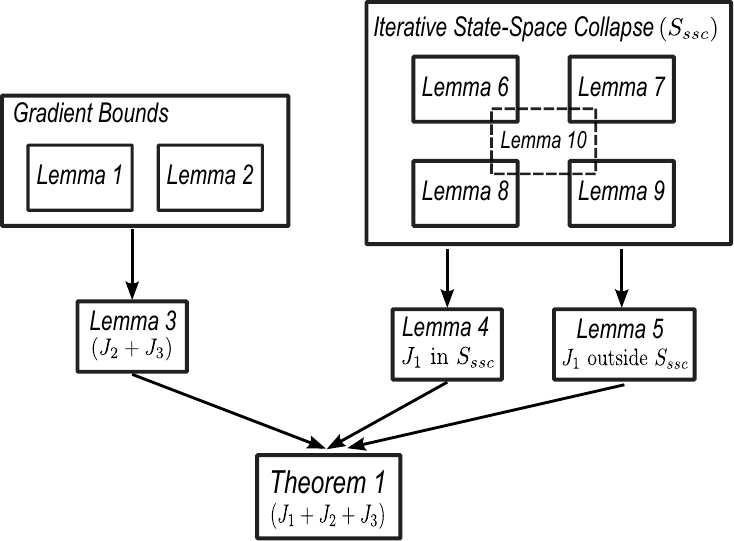}
  \caption{~The roadmap of proving Theorem \ref{Thm:main}.}
  \label{fig:roadmap}
\end{figure}

\subsection{Gradient Bounds}\label{sec:Grad_Bounds}
To bound $J_2$ in \eqref{G-expansion-Gradient-1} and $J_3$ in \eqref{G-expansion-Gradient-2}, we summarize bounds on $g'$ and $g''$ in the following two lemmas.

\begin{lemma}\label{lemma:g'}
Given  $x\in\left[ \eta  -\frac{2}{N}, \eta  +\frac{2}{N}\right],$ we have $$|g'(x)|\leq \frac{2}{\sqrt{N}\log N}.$$ \qed 
\end{lemma}

\begin{lemma}\label{lemma:g''}
For $x>\eta,$ we have 
\begin{align*}
|g''(x)|\leq \frac{\sqrt{N}}{\log N}.  
\end{align*}
\qed 
\end{lemma}

Based on the bounds on $g'$ in Lemma \ref{lemma:g'} and $g''$ in Lemma \ref{lemma:g''}, we provide the upper bound on $J_2+J_3$ in the following lemma.

\begin{lemma}\label{Gradient:final}
For $g(\cdot)$ defined in \eqref{eq:L}, we have
\begin{align*}
J_2 + J_3 \leq  \frac{6\mu_{\max}}{\sqrt{N}\log N}. 
\end{align*}
\qed
\end{lemma}

The proofs of the lemmas above are presented in Appendix \ref{appendix:GB}. 

\subsection{State Space Collapse (SSC)}\label{sec:SSC}

In this subsection, we analyze $J_1$ in \eqref{G-expansion-SSC}:
\begin{align}
&\mathbb E\left[g'\left(\sum_{i=1}^b S_{i}\right)\left(\lambda A_b(S) - \lambda-\frac{\log N}{\sqrt{N}}+D_1\right)\mathbb{I}_{\sum_{i=1}^b S_{i} > \eta + \frac{1}{N}}\right] \nonumber\\
=& \mathbb E\left[\frac{\sqrt{N}}{\log N} h\left(\sum_{i=1}^bS_{i}\right) \left(-\lambda A_b(S) + \lambda + \frac{\log N}{\sqrt{N}}-D_1\right)\mathbb{I}_{\sum_{i=1}^b S_{i} > \eta + \frac{1}{N}}\right]\nonumber\\
\leq& \mathbb E\left[\frac{\sqrt{N}}{\log N} h\left(\sum_{i=1}^b S_{i}\right) \left(\lambda + \frac{\log N}{\sqrt{N}}-D_1\right)\mathbb{I}_{\sum_{i=1}^b S_{i} > \eta + \frac{1}{N}}\right]\label{SSC0},
\end{align} where the equality is due to Stein's equation \eqref{Gen:L}, and the inequality holds because $$\frac{\sqrt{N}}{\log N} h\left(\sum_{i=1}^b S_{i}\right) \mathbb{I}_{\sum_{i=1}^b S_{i} > \eta + \frac{1}{N}}\geq 0.$$

We first focus on 
\begin{equation}
\left(\lambda + \frac{\log N}{\sqrt{N}}-(1-p)\mu_1s_{1,1}-\mu_2s_{1,2}\right)\mathbb{I}_{\sum_{i=1}^b s_{i} > \eta + \frac{1}{N}},\label{eq:sscterm}
\end{equation}
where we recall $\eta = \lambda +\frac{k\log N}{\sqrt{N}}$ and $d_1=(1-p)\mu_1s_{1,1}+\mu_2s_{1,2}$ is the total departure rate when the system is in the state $s.$

We consider two cases: $s\in \mathcal S_{ssc} $ and $s\not\in \mathcal S_{ssc},$ where $$\mathcal S_{ssc} = \mathcal S_{ssc_1} \bigcup S_{ssc_2},$$ and
\begin{align*}
\mathcal S_{ssc_1} =& \left\{s ~\left|~ s_1 \geq \lambda + \left(\frac{1+\mu_1+\mu_2}{w_l}-\mu_1\right)\frac{\log N}{\sqrt{N}}, \right. \right. \\ 
 &\left. \left.  ~~~~~~s_{1,1} \geq \frac{\lambda}{\mu_1} - \frac{\log N}{\sqrt{N}}, \text{and}~ s_{1,2} \geq \frac{p\lambda}{\mu_2} - \frac{\mu_1\log N}{\sqrt{N}}\right.\right\}, \\
\mathcal S_{ssc_2} =& \left\{s ~\left|~ \sum_{i=1}^{b} s_i \leq \lambda+\frac{k\log N}{\sqrt{N}}\right. \right\}.
\end{align*}
\begin{itemize}
    \item {\bf Case 1:} $\mathcal S_{ssc_1}$ is shown as the gray region in Figure~\ref{fig:final SSC}. Any $s\in \mathcal S_{ssc_1}$ satisfies $$(1-p)\mu_1s_{1,1}+\mu_2s_{1,2} \geq \lambda + \frac{\log N}{\sqrt{N}},$$ so  $\left(\lambda + \frac{\log N}{\sqrt{N}}-(1-p)\mu_1s_{1,1}-\mu_2s_{1,2}\right)\mathbb{I}_{\sum_{i=1}^b s_{i} > \eta + \frac{1}{N}} \leq 0$ for any $s\in \mathcal S_{ssc_1}.$ The details are presented in Lemma \ref{SSC:s1 s2 negative}.   When $s\in \mathcal S_{ssc_2},$ $$\mathbb{I}_{\sum_{i=1}^b s_{i} > \eta + \frac{1}{N}}=0,$$ so $\left(\lambda + \frac{\log N}{\sqrt{N}}-(1-p)\mu_1s_{1,1}-\mu_2s_{1,2}\right)\mathbb{I}_{\sum_{i=1}^b s_{i} > \eta + \frac{1}{N}} =0$ for any $s\in \mathcal S_{ssc_2}.$ 
    
    \item {\bf Case 2:} We will show that $$\mathbb P\left( S \notin \mathcal S_{ssc}\right) \leq \frac{3}{N^{2}}$$ in Lemma \ref{SSC:out} using an iterative state space collapse approach. 
\end{itemize}

\begin{figure}[!htbp]
  \centering
  \includegraphics[width=3.1in]{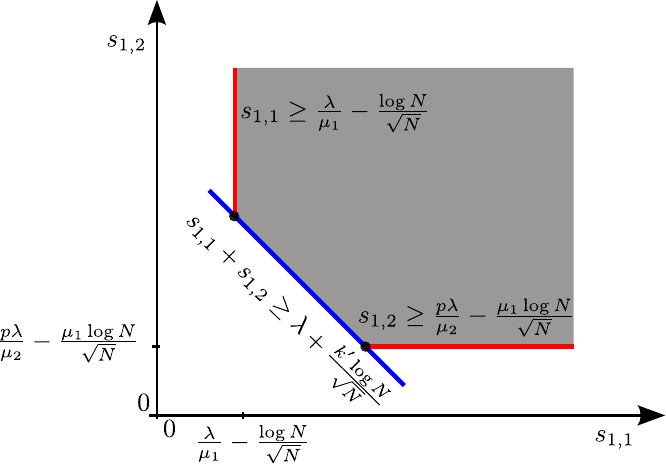}
  \caption{~State Space Collapse in $\mathcal S_{ssc_1}.$}
  \label{fig:final SSC}
\end{figure}

\begin{lemma}\label{SSC:s1 s2 negative} For any $s \in \mathcal S_{ssc_1},$ 
\begin{align*} 
\left(\lambda+\frac{\log N}{\sqrt{N}}-(1-p)\mu_1s_{1,1}-\mu_2s_{1,2}\right)\mathbb{I}_{\sum_{i=1}^bs_i> \lambda +\frac{k\log N}{\sqrt{N}}+\frac{1}{N}} \leq 0.
\end{align*} \hfill{$\square$}
\end{lemma}
The proof of Lemma \ref{SSC:s1 s2 negative} can be found in Appendix \ref{app:inssc}.

\begin{lemma}\label{SSC:out}
For a large $N$ such that $\log N \geq \frac{3.5}{\min\left(\frac{\mu_1}{16\mu_{\max}},\frac{\mu_2}{12\mu_{\max}},\frac{\mu_1\mu_2}{40\mu_{\max}}\right)},$ we have
$$\mathbb P\left( S \notin \mathcal S_{ssc}\right) \leq \frac{3}{N^{2}}.$$  \hfill{$\square$}
\end{lemma}

\begin{proof}
The proof of Lemma \ref{SSC:out} is based on an ``iterative" procedure to establish state space collapse, which is achieved by proving a sequence of lemmas (Lemma \ref{tailbound-u:s12} - Lemma \ref{SSC:s1 s2}). The detailed proof of four lemmas can be found in Appendix \ref{App:IterSSC}.

Define sets $\tilde{\cal S}_1$ and $\tilde{\cal S}_2$  such that 
\begin{align}
\tilde{\cal S}_1=&\left\{s\left|s_{1,1} \geq \frac{\lambda}{\mu_1} - \frac{\log N}{\sqrt{N}}\hbox{ and } s_{1,2}\geq \frac{p\lambda}{\mu_2} - \frac{\mu_1\log N}{\sqrt{N}}\right.\right\}\\
\tilde{\cal S}_2=&\left\{s\left| \min\left\{\eta-s_1,\sum_{i=2}^{b} s_i\right\} \leq \frac{(c_1+\mu_1)\log N}{\sqrt{N}}\right.\right\}. 
\end{align}
According to the union bound and Lemmas \ref{tailbound-l:s11}-\ref{SSC:s1 s2}, we have 
\begin{align*}
&\mathbb P \left( S \notin \tilde{\mathcal S}_1\cap \tilde{\mathcal S}_2\right)\\ 
\leq& \frac{5}{\mu_1}\frac{\sqrt{N}}{\log N} e^{-\min\left(\frac{\mu_1}{16\mu_{\max}},\frac{\mu_1\mu_2}{40\mu_{\max}}\right)\log^2 N}+\frac{16}{\mu_1\mu_2} \frac{N}{\log^2 N} e^{-\min\left(\frac{\mu_1}{16\mu_{\max}},\frac{\mu_2}{12\mu_{\max}},\frac{\mu_1\mu_2}{40\mu_{\max}}\right)\log^2 N},\\
&+\frac{34}{\mu_1^2\mu_2} \frac{N^{1.5}}{\log^3 N} e^{-\min\left(\frac{\mu_1}{16\mu_{\max}},\frac{\mu_2}{12\mu_{\max}},\frac{\mu_1\mu_2}{40\mu_{\max}}\right)\log^2 N}\\
\leq& \frac{3}{N^{2}},    
\end{align*}
where the second inequality holds for a sufficiently large $N$ such that $$\log N \geq 
\frac{3.5}{\min\left(\frac{\mu_1}{16\mu_{\max}},\frac{\mu_2}{12\mu_{\max}},\frac{\mu_1\mu_2}{40\mu_{\max}}\right)}.$$

We note that  $\tilde{\cal S}_1\cap \tilde{\cal S}_2 $  is a subset of  ${\cal S}_{ssc}.$ This is because for any $s$ which satisfies 
$$\min\left\{\eta-s_1,\sum_{i=2}^{b} s_i\right\} \leq \frac{(c_1+\mu_1)\log N}{\sqrt{N}},$$ we either have $$\eta-s_1\leq \frac{(c_1+\mu_1)\log N}{\sqrt{N}},$$ which implies 
$$s_1\geq \lambda+\left(\frac{1+\mu_1+\mu_2}{w_l}-\mu_1\right)\frac{\log N}{\sqrt{N}};$$ or 
$$\sum_{i=2}^{b} s_i\leq \eta-s_1,$$ which implies 
$$\sum_{i=1}^{b} s_i\leq \eta.$$ Note that $$\tilde{\cal S}_1\cap\left\{s\left|s_1\geq \lambda+\left(\frac{1+\mu_1+\mu_2}{w_l}-\mu_1\right)\frac{\log N}{\sqrt{N}}\right.\right\}={\cal S}_{ssc_1}$$ and 
$$\tilde{\cal S}_1\cap\left\{s\left|\sum_{i=1}^{b} s_i\leq \eta\right.\right\}\subseteq {\cal S}_{ssc_2}.$$ We, therefore, have  $$\tilde{\cal S}_1\cap \tilde{\cal S}_2\subseteq {\cal S}_{ssc},$$ and \begin{align*}
\mathbb P\left( S \notin {\mathcal S}_{ssc}\right) \leq  \mathbb P\left( S \notin \tilde{\mathcal S}_1\cap \tilde{\mathcal S}_2\right)
\leq \frac{3}{N^{2}},
\end{align*} so Lemma \ref{SSC:out} holds.
\end{proof}

We present ``iterative" state space collapse procedure in Lemma \ref{tailbound-u:s12} - Lemma \ref{SSC:s1 s2}.
\begin{lemma}[An Upper Bound on $S_{1,2}$]\label{tailbound-u:s12}
$$\mathbb P\left(S_{1,2} \leq \frac{p}{\mu_2} + \frac{\log N}{2\sqrt{N}}\right) \geq 1- e^{-\frac{\mu_1\mu_2\log^2 N}{40\mu_{\max}}}.$$ \qed
\end{lemma}

\begin{lemma}[A Lower Bound on $S_{1,1}$]\label{tailbound-l:s11}
$$\mathbb P\left(S_{1,1} \geq \frac{\lambda}{\mu_1} - \frac{\log N}{\sqrt{N}}\right) \geq 1-\frac{5}{\mu_1}\frac{\sqrt{N}}{\log N} e^{-\min\left(\frac{\mu_1}{16\mu_{\max}},\frac{\mu_1\mu_2}{40\mu_{\max}}\right)\log^2 N}.$$ \qed
\end{lemma}

\begin{lemma}[A Lower Bound on $S_{1,2}$]\label{tailbound-l:s12}
$$\mathbb P\left(S_{1,2} \geq \frac{p\lambda}{\mu_2} - \frac{\mu_1\log N}{\sqrt{N}}\right) \geq 1-\frac{16}{\mu_1\mu_2} \frac{N}{\log^2 N} e^{-\min\left(\frac{\mu_1}{16\mu_{\max}},\frac{\mu_2}{12\mu_{\max}},\frac{\mu_1\mu_2}{40\mu_{\max}}\right)\log^2 N}.$$ \qed
\end{lemma}

\begin{lemma}[A Lower Bound on $S_1$ via $\sum_{i=2}^b S_i$]\label{SSC:s1 s2}
\begin{align*}
&\mathbb P\left(\min\left\{\lambda +\frac{k\log N}{\sqrt{N}}-S_1,\sum_{i=2}^{b} S_i\right\} \leq \frac{(c_1+\mu_1)\log N}{\sqrt{N}}\right) \\
&\geq 1-\frac{34}{\mu_1^2\mu_2} \frac{N^{1.5}}{\log^3 N} e^{-\min\left(\frac{\mu_1}{16\mu_{\max}},\frac{\mu_2}{12\mu_{\max}},\frac{\mu_1\mu_2}{40\mu_{\max}}\right)\log^2 N}\end{align*} 
for $\log N \geq \frac{1}{\min\{\mu_1,\mu_2\}},$ where $k = \left(1+\frac{w_u b}{w_l}\right)\left(\frac{1+\mu_1+\mu_2}{w_l}+2\mu_1\right)$ and $c_1 = \frac{w_u b}{w_l}\left(\frac{1+\mu_1+\mu_2}{w_l}+2\mu_1\right)+2\mu_1.$ \qed
\end{lemma}

{\bf Remark:} An important contribution of this paper is the iterative state collapse method we use to prove Lemma \ref{SSC:out}. The method continues refining the state space in which the system stays with a high probability at steady-state.  Figure~\ref{BoundS1-issc} illustrates the iterative state-space collapse in Lemma \ref{tailbound-u:s12} - Lemma \ref{tailbound-l:s12}. We first show in Lemma \ref{tailbound-u:s12} that with a high probability, $S_{1,2}\leq \frac{p}{\mu_2}+\frac{\log N}{2\sqrt{N}}$ at steady-state. Then in the reduced state space $\left(S_{1,2}\leq \frac{p}{\mu_2}+\frac{\log N}{2\sqrt{N}}\right)$, we further show in Lemma \ref{tailbound-l:s11} that $S_{1,1}\geq \frac{\lambda}{\mu_1}-\frac{\log N}{\sqrt{N}}$ with a high probability at steady state. We then further establish in Lemma \ref{tailbound-u:s12} that $S_{1,2}\geq \frac{p\lambda}{\mu_2}-\frac{\mu_1\log N}{\sqrt{N}}$ with a high probability at steady state in the reduced state space. 

\begin{figure}[!htbp]
  \centering
  \includegraphics[width=6.1in]{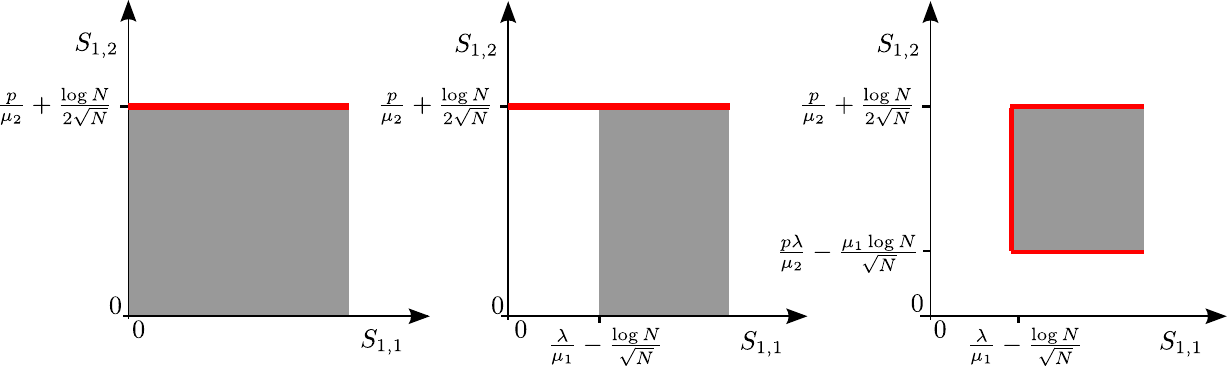}
  \caption{~Iterative State-Space Collapse to Show that  $S_{1,1}$ and $S_{1,2}$ are in a Smaller State-Space (the Gray Region) at Steady-State}
  \label{BoundS1-issc}
\end{figure}

\subsection{Proof of Theorem \ref{Thm:main}}
Based on Lemma \ref{SSC:s1 s2 negative} and Lemma \ref{SSC:out}, we can establish the following bound on \eqref{SSC0}, which is a upper bound on $J_1$ in \eqref{G-expansion-SSC},
\begin{align}
&\mathbb E\left[\frac{\sqrt{N}}{\log N} h\left(\sum_{i=1}^b\sum_{m=1}^2 S_{i,m}\right) \left(\lambda + \frac{\log N}{\sqrt{N}}-D_1\right)\mathbb{I}_{\sum_{i=1}^b S_{i} > \eta + \frac{1}{N}}\right] \notag\\
=&\mathbb E\left[\frac{\sqrt{N}}{\log N}\left(\sum_{i=1}^bS_i-\eta\right)\left(\lambda+\frac{\log N}{\sqrt{N}}-D_1\right)\mathbb{I}_{S \in \mathcal S_{ssc}}\mathbb{I}_{\sum_{i=1}^bS_i> \eta +\frac{1}{N}}\right]\nonumber\\
&+\mathbb E\left[\frac{\sqrt{N}}{\log N}\left(\sum_{i=1}^bS_i-\eta\right)\left(\lambda+\frac{\log N}{\sqrt{N}}-D_1\right)\mathbb{I}_{S \notin \mathcal S_{ssc}}\mathbb{I}_{\sum_{i=1}^bS_i> \eta +\frac{1}{N}}\right] \nonumber\\
\leq& \frac{3b}{N^{1.5}\log N},  \label{SSC-bound}
\end{align}
where the last inequality holds because we have used the facts that the average total number of jobs per server is at most $b$ and $\left(\lambda+\frac{\log N}{\sqrt{N}}-D_1\right)\mathbb{I}_{S \notin \mathcal S_{ssc}}\mathbb{I}_{\sum_{i=1}^bS_i> \eta +\frac{1}{N}} < 1.$

Based on Lemma \ref{Gradient:final}, we are ready to establish Theorem \ref{Thm:main} under JSQ.
\begin{align}
&\mathbb E\left[\max\left\{\sum_{i=1}^{b} S_i- \eta,0\right\}\right] = J_1 + J_2 + J_3 \leq\frac{3b}{N^{1.5}\log N}+\frac{6\mu_{\max}}{\sqrt{N}\log N}, \nonumber
\end{align} 
which implies $$\mathbb E\left[\max\left\{\sum_{i=1}^{b} S_i- \eta,0\right\}\right] \leq \frac{7\mu_{\max}}{\sqrt{N}\log N}.$$

\section{Conclusions}
In this paper, we considered load balancing under the Coxian-2 service time distribution in heavy traffic regimes. The Coxian-2 service time distribution does not have DHR and the system considered in this paper lacks monotonicity. We developed an iterative SSC and identified a policy set $\Pi,$ in which any policy can achieve asymptotic zero delay. The set $\Pi$ includes JSQ, JIQ, I1F and Po$d$ with $d=O\left(\frac{\log N}{1-\lambda}\right).$
The proposed Stein's method with iterative SCC is a general method that can be used for steady-state analysis of other queueing systems. The key idea of this method is to use an iterative SSC to reduce the state space to a much smaller subspace, in which the system can be well approximated with a simple fluid model, and the approximation error can be quantified using Stein's method. The iterative SSC approach iteratively reduces the state space by focusing on one direction at each iteration based on the system dynamics. This provides an intuitive way to establish SSC results that may be difficult to obtain at once. For example, it remains open whether the SSC result in this paper can be proved using a single Lyapunov function. This method has already inspired and been used in recent work \cite{WenZhoSri_20},
which developed zero-delay load balancing algorithms for networked servers assuming exponential service times.

We also would like to remark it is nontrivial to extend the results in this paper beyond Coxian-$2.$ The analysis in this paper utilized some simple yet critical properties of the Coxian-2 distribution: a job in phase-$1$ either departs or enters phase-$2$ immediately, and a job always starts its service from phase-$1$. In a Coxian-$M$ distribution or a general phase-type distribution, the dependence between jobs in different phases becomes more involved. In particular, it becomes more challenging to establish a result similar to Lemma \ref{tailbound-u:s12}. Recall that for a Coxian-2 distribution, $s_{1,2}$ decreases when its value is large because a large $s_{1,2}$ implies $s_{1,1}$ is small (because $s_{1,1}+s_{1,2}\leq 1$)  so the rate at which jobs move from phase-1 to phase-2 is  small. However,   for a Coxian-$M$ distribution, a large $s_{1,M}$ is not sufficient to guarantee that $s_{1,M-1}$ is small enough so that $s_{1,M}$ will decrease. For a general phase-type distribution, jobs in the queues can be in any phase, not necessarily in phase-1, which makes it difficult to show that $S_{1,M}$ will be close to its ``equilibrium''. However, we believe if a proper Lyapunov function could be found to establish a ``good'' upper bound on $S_{1,M},$ then we may apply the iterative approach in this paper to establish SSC and to extend the results in this paper to more general service distributions.

\section*{Acknowledgements} 
The authors are very grateful to Prof. Jim Dai for his insightful comments. The discussions with Jim had continuously stimulated the authors during the writing of this paper. This work was supported in part by NSF ECCS 1739344, CNS 2002608 and CNS 2001687. 

\bibliographystyle{abbrv}
\bibliography{inlab-refs}%

\appendix

\section{Gradient Bounds} \label{appendix:GB}
\subsection{Proof of Lemma \ref{lemma:g'}}
\begin{proof}
From the definition of $g$ function in \eqref{Gen:L}, we have
$$g'(x)=\frac{\max\left\{x-\eta, 0\right\}}{-\frac{\log N}{\sqrt{N}}}.$$
Hence, for any $x\in\left[ \eta  -\frac{2}{N}, \eta  +\frac{2}{N}\right],$ we have
$$|g'(x)| \leq \frac{|x-\eta|}{\frac{\log N}{\sqrt{N}}}\leq \frac{\frac{2}{N}}{\frac{\log N}{\sqrt{N}}}=\frac{2}{\sqrt{N}\log N}.$$
\end{proof}

\subsection{Proof of Lemma \ref{lemma:g''}}
\begin{proof}
From the definition of $g$ function in \eqref{Gen:L}, we have
$$g'(x)=\frac{\max\left\{x-\eta, 0\right\}}{-\frac{\log N}{\sqrt{N}}}.$$
For $x> \eta,$ we have
$$g'(x)=\frac{x-\eta}{-\frac{\log N}{\sqrt{N}}},$$ which implies
\begin{eqnarray*}
|g''(x)|=\left|\frac{1}{-\frac{\log N}{\sqrt{N}}}\right|=\frac{\sqrt{N}}{\log{N}}.
\end{eqnarray*}
\end{proof}

\subsection{Proof of Lemma \ref{Gradient:final}}

Note $(1-p)\mu_1 s_{1,1}+\mu_2s_{1,2} \leq \mu_{\max} s_{1} \leq \mu_{\max},$ then we have
\begin{align}
J_2 + J_3 
\leq &\mathbb E\left[\left(g'\left(\sum_{i=1}^{b} S_i\right)\left(-\frac{\log N}{\sqrt{N}}\right)+\lambda|g'(\xi)|+\mu_{\max}|g'(\tilde{\xi})|\right)\mathbb{I}_{\sum_{i=1}^{b} S_i \in \mathcal T_1}\right] \\
&+ \mathbb E\left[\frac{1}{N}\left(\lambda |g''(\eta)|  + \mu_{\max} |g''(\tilde{\eta})| \right) \mathbb{I}_{\sum_{i=1}^{b} S_i \in \mathcal T_2}\right] \\
\leq & \frac{4\mu_{\max}}{\sqrt{N}\log N} + \frac{\lambda+\mu_{\max}}{N} \frac{\sqrt{N}}{\log N} \\
\leq & \frac{6\mu_{\max}}{\sqrt{N}\log N}
\end{align}

\def \SSCsonestwonegative{\ref{SSC:s1 s2 negative}}
\section{Proof of Lemma \SSCsonestwonegative}
\label{app:inssc}

We consider the following problem $$\min_{(s_{1,1}, s_{1,2})\in \mathcal S_{ssc_1}}(1-p)\mu_1 s_{1,1}+\mu_2 s_{1,2},$$ which is a linear programming in terms of variables $s_{1,1}$ and $s_{1,2}$. Therefore, we only need to consider the extreme points of set $\mathcal S_{ssc_1}.$ In fact, from Figure~\ref{fig:final SSC}, it is clear that we only need to consider the following two extreme points. 

\begin{itemize}
\item Case $1$: $s_{1,1}=\frac{\lambda}{\mu_1}-\frac{\log N}{\sqrt{N}}$ and $s_{1,2}=\lambda +\left(\frac{1+\mu_1+\mu_2}{w_l}-\mu_1\right)\frac{\log N}{\sqrt{N}}-s_{1,1}=\frac{p\lambda}{\mu_2}+\left(\frac{1+\mu_1+\mu_2}{w_l}-\mu_1+1\right)\frac{\log N}{\sqrt{N}},$ where we use the fact $\frac{1}{\mu_1}+\frac{p}{\mu_2}=1.$ 
In this case, 
\begin{align}
(1-p)\mu_1s_{1,1}+\mu_2s_{1,2} 
=& \lambda + \left(-(1-p)\mu_1+\mu_2\left(\frac{1+\mu_1+\mu_2}{w_l}-\mu_1+1\right)\right) \frac{\log N}{\sqrt{N}}\\
\geq& \lambda + \left(-(1-p)\mu_1+\left(1+\mu_1-\mu_1\mu_2+2\mu_2\right)\right) \frac{\log N}{\sqrt{N}}\label{SSCCase1: negative 1}\\
=&\lambda + \left(1+\mu_2\right) \frac{\log N}{\sqrt{N}}\label{SSCCase1: negative 2}\\
\geq& \lambda + \frac{\log N}{\sqrt{N}} \label{SSCCase1: negative 5},
\end{align}
where \eqref{SSCCase1: negative 1} holds because $w_l=\min\{(1-p)\mu_1, \mu_2\}$ and \eqref{SSCCase1: negative 2} holds because $\frac{1}{\mu_1}+\frac{p}{\mu_2}=1.$

\item Case 2: $s_{1,1}=\lambda +\left(\frac{1+\mu_1+\mu_2}{w_l}-\mu_1\right)\frac{\log N}{\sqrt{N}}-s_{1,2}=\frac{\lambda}{\mu_1} +\frac{1+\mu_1+\mu_2}{w_l}\frac{\log N}{\sqrt{N}}$ and $s_{1,2}=\frac{p\lambda}{\mu_2}-\frac{\mu_1\log N}{\sqrt{N}}.$ At this extreme point, we have
\begin{align}
(1-p)\mu_1s_{1,1}+\mu_2s_{1,2} 
=& \lambda + \left((1-p)\mu_1\left(\frac{1+\mu_1+\mu_2}{w_l}\right)-\mu_1\mu_2\right) \frac{\log N}{\sqrt{N}} \label{SSCCase2: negative 3}\\
\geq& \lambda + \left({1+\mu_1+\mu_2}-\mu_1\mu_2\right) \frac{\log N}{\sqrt{N}} \label{SSCCase2: negative 4}\\
\geq& \lambda + \frac{\log N}{\sqrt{N}}, \label{SSCCase2: negative 5}
\end{align}  where 
\eqref{SSCCase2: negative 4} holds because $w_l=\min\{(1-p)\mu_1, \mu_2\}$ and  \eqref{SSCCase2: negative 5} holds because $\mu_1 +\mu_2 \geq p \mu_{1} + \mu_2 = \mu_1\mu_2.$
\end{itemize}

\def \SSCout{\ref{SSC:out}~}
\section{Proof of Iterative State Space Collapse} \label{App:IterSSC}

We present the iterative SSC approach for proving Lemma \ref{tailbound-u:s12}-Lemma \ref{SSC:s1 s2}. The first three lemmas are on the upper and lower bounds on $S_{1,1}$ and $S_{1,2},$ illustrated in Figure~ \ref{BoundS1}, which shows that  both $S_{1,1}$ and $S_{1,2}$ are close to its equilibrium values, in particular, with a high probability, 
$S_{1,1} \geq \frac{\lambda}{\mu_1}-\frac{\log N}{\sqrt{N}}$ and $S_{1,2} \geq \frac{p\lambda}{\mu_2}-\frac{\mu_1\log N}{\sqrt{N}}.$  
However, these two low bounds do not guarantee the total departure rate, which is $(1-p)\mu_1S_{1,1}+\mu_2S_{1,2},$ is larger than the arrival rate $\lambda.$
Therefore, we need Lemma \ref{SSC:s1 s2} to guarantee sufficient fraction of busy servers $S_1$ such that the total departure rate is "larger than" the arrival rate $\lambda.$ We therefore need Lemma \ref{SSC:s1 s2} to further establish a lower bound on $S_1$ unless the total normalized queue length $\sum_{i=1}^b S_i$ is small. 

\begin{figure}[!htbp]
  \centering
  \includegraphics[width=5.0in]{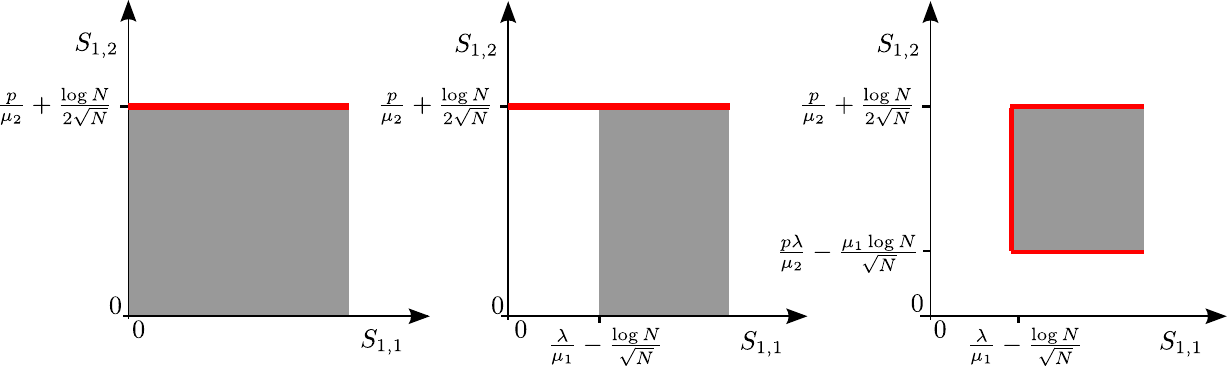}
  \caption{~Bounds (red lines) on $S_{1,1}$ and $S_{1,2}.$}
  \label{BoundS1}
\end{figure}

\subsection[]{A tail bound from \cite{WanMagSri_18}}

To prove the space space collapse results, we first introduce Lemma \ref{tail-bound-cond}, which will be repeatedly used to obtain probability tail bounds. Lemma \ref{tail-bound-cond} allows us to apply Lyapunov-drift-based heavy traffic analysis \cite{ErySri_12} to reduced state spaces instead of to the entire state space. The lemma is an extension of the tail bound in \cite{BerGamTsi_01}. This Lyapunov drift analysis on reduced state space enables us to iteratively refine the state space at steady state. The lemma was proven in \cite{WanMagSri_18}. We include the proof to make the paper self-contained. 

\begin{lemma}\label{tail-bound-cond}
Let $(S(t): t \geq 0)$ be a continuous-time Markov chain over a finite state space $\mathcal S$ and is irreducible, so it has a unique stationary distribution $\pi.$  Consider a Lyapunov function $V: \mathcal S \to R^{+}$ and define the drift of $V$ at a state $s \in \mathcal S$ as $$\nabla V(s) = \sum_{s' \in \mathcal S: s' \neq s} q_{s,s'} (V(s') - V(s)),$$ where $q_{s,s'}$ is the transition rate from $s$ to $s'.$ Assume
\begin{align*} 
\nu_{\max} :=& \max\limits_{s,s'\in \mathcal S: q_{s,s'} >0} |V(s') - V(s)|< \infty ~~ \text{and} ~~ \bar q := \max\limits_{s \in \mathcal S} (-q_{s,s}) < \infty 
\end{align*}
and define $$q_{\max}:=\max\limits_{s \in \mathcal S} \sum_{s' \in \mathcal S: V(s) < V(s')} q_{s,s'}.$$
If there exists a set $\mathcal E$ with $B>0$, $\gamma>0$, $\delta \geq 0$ such that the following conditions satisfy: 
\begin{enumerate}[(i)]
\item $\nabla V(s) \leq -\gamma$ when $V(s) \geq B$ and $s \in \mathcal E.$
\item $\nabla V(s) \leq \delta$ when $V(s) \geq B$ and $s \notin \mathcal E.$
\end{enumerate}
Then $$\mathbb P \left(V(S)\geq B + 2 \nu_{\max} j \right) \leq \alpha^j + \beta \mathbb P\left(S \notin \mathcal E\right), ~ \forall j \in \mathbb{N},$$
with $$\alpha = \frac{q_{\max}\nu_{\max}}{q_{\max}\nu_{\max} + \gamma} ~~\text{and}~~ \beta = \frac{\delta}{\gamma}+1.$$
\end{lemma}

\begin{proof}
Let $C \geq B - \nu_{\max}$ and consider Lyapunov function $$\hat V(s) = \max\{C, V(s)\}.$$ 

At steady state, we have 
\begin{align}
0 =& \sum_{ V(s) \leq C - \nu_{\max}} \pi(s) \sum_{s'\neq s} q_{s,s'}\left(\hat V(s')-\hat V(s)\right) \nonumber\\
   &+ \sum_{C - \nu_{\max} <  V(s) \leq C + \nu_{\max}} \pi(s) \sum_{s'\neq s} q_{s,s'}\left(\hat V(s')-\hat V(s)\right) \nonumber\\
   &+ \sum_{V(s)>C + \nu_{\max}} \pi(s) \sum_{s'\neq s} q_{s,s'}\left(\hat V(s')-\hat V(s)\right). \label{Drift_equilibrium}
\end{align}

Note $\nabla \hat V(s) = \sum_{s'\neq s} q_{s,s'}\left(\hat V(s')-\hat V(s)\right)$. We consider three terms in \eqref{Drift_equilibrium} as follows:

\begin{itemize}
\item The first term is $0$ because $V(s) \leq C - \nu_{\max}$ and $V(s') \leq C$ imply $\hat V(s) = \hat V(s') = C$.

\item The second term is bounded
\begin{align*}
&\sum_{C - \nu_{\max} <  V(s) \leq C + \nu_{\max}} \pi(s) \sum_{s'\neq s} q_{s,s'}\left(\hat V(s')-\hat V(s)\right) \\
\leq& \sum_{C - \nu_{\max} <  V(s) \leq C + \nu_{\max}} \pi(s) q_{\max} \nu_{\max} \\
\leq& q_{\max} \nu_{\max} \left( \mathbb P( V(S) > C - \nu_{\max}) - \mathbb P( V(S) > C + \nu_{\max}) \right)
\end{align*}

\item The third term is divided into two regions $s \in \mathcal E$ and $s \notin \mathcal E$ 
\begin{align*}
&\sum_{V(s)>C + \nu_{\max}} \pi(s) \sum_{s'\neq s} q_{s,s'}\left(\hat V(s')-\hat V(s)\right) \\
=& \mathop{\sum_{V(s)>C + \nu_{\max}}}_{s \in \mathcal E} \pi(s) \sum_{s'\neq s} q_{s,s'}\left(\hat V(s')-\hat V(s)\right) + \mathop{\sum_{V(s)>C + \nu_{\max}}}_{s \notin \mathcal E} \pi(s) \sum_{s'\neq s} q_{s,s'}\left(\hat V(s')-\hat V(s)\right) \\
\leq& -\gamma \mathbb P\left(V(S)>C + \nu_{\max}, s \in \mathcal E\right) + \delta \mathbb P\left(V(S)>C + \nu_{\max}, s \notin \mathcal E\right)\\
=& -\gamma \mathbb P\left(V(S)>C + \nu_{\max}\right) + (\delta + \gamma) \mathbb P\left(V(S)>C + \nu_{\max}, s \notin \mathcal E\right)
\end{align*}
where the inequality holds because of two conditions (i) and (ii).
\end{itemize}

Combining three terms above, we have 
\begin{align*}
&(q_{\max} \nu_{\max} + \gamma) \mathbb P( V(S) > C + \nu_{\max}) \\
\leq& q_{\max} \nu_{\max}  \mathbb P( V(S) > C - \nu_{\max}) + (\delta + \gamma) \mathbb P\left(V(S)>C + \nu_{\max}, S \notin \mathcal E\right)
\end{align*}
which implies
\begin{align*}
&\mathbb P( V(S) > C + \nu_{\max}) \\
\leq& \frac{q_{\max} \nu_{\max}}{q_{\max} \nu_{\max} + \gamma}  \mathbb P( V(S) > C - \nu_{\max}) + \frac{\delta + \gamma}{q_{\max} \nu_{\max} + \gamma} \mathbb P\left(V(S)>C + \nu_{\max}, S \notin \mathcal E\right) \\
\leq& \frac{q_{\max} \nu_{\max}}{q_{\max} \nu_{\max} + \gamma}  \mathbb P( V(S) > C - \nu_{\max}) + \frac{\delta + \gamma}{q_{\max} \nu_{\max} + \gamma} \mathbb P\left(S \notin \mathcal E \right) \\
=&  \alpha \mathbb P( V(S) > C - \nu_{\max}) + \kappa \mathbb P\left(S \notin \mathcal E\right),
\end{align*}
where $$\alpha = \frac{q_{\max}\nu_{\max}}{q_{\max}\nu_{\max} + \gamma} ~~\text{and}~~ \kappa = \frac{\delta+\gamma}{q_{\max}\nu_{\max} + \gamma}.$$

Let $C = B + (2j -1)\nu_{\max}, \forall j \in \mathbb{N}$ and we have 
\begin{align}
&\mathbb P\left( V(S) > B + 2\nu_{\max}j\right) \nonumber\\
\leq&  \alpha \mathbb P\left( V(S) > B + 2(j-1)\nu_{\max}\right) + \kappa \mathbb P\left(S \notin \mathcal E\right) \label{recursion}
\end{align}

By recursively using the inequality \eqref{recursion}, we have
\begin{align*}
\mathbb P\left( V(S) > B + 2\nu_{\max}j\right) 
\leq&  \alpha^j + \kappa \mathbb P\left(S \notin \mathcal E\right) \sum_{i=0}^{j} \alpha^i  \\
\leq&  \alpha^j + \frac{\kappa}{1 - \alpha} \mathbb P\left(S \notin \mathcal E\right) \\
=&  \alpha^j + \beta \mathbb P\left(S \notin \mathcal E\right) 
\end{align*}
\end{proof}

As mentioned above, Lemma \ref{tail-bound-cond} is an extension of Theorem 1 in \cite{BerGamTsi_01}, where $\mathcal E = \mathcal S^{(N)}$ is the entire state space and $\mathbb P\left(S \notin \mathcal E\right)=0.$
As suggested in Lemma \ref{tail-bound-cond}, constructing proper Lyapunov functions are critical to establish the tail bounds. In the following lemmas, we construct a sequence of Lyapunov functions and apply Lemma \ref{tail-bound-cond} to establish SSC results. 
 
\subsection{Proof of Lemma \ref{tailbound-u:s12}: An upper bound on $S_{1,2}.$}

To prove Lemma \ref{tailbound-u:s12}, we first establish a Lyaponuv drift analysis for $\mathcal E = \mathcal S^{(N)}$ (the entire state space) in Lemma \ref{driftbound-u:s12}.

\begin{lemma}\label{driftbound-u:s12}
Consider Lyapunov function $$V(s) = s_{1,2} - \frac{p}{\mu_2}.$$ When $V(s) \geq \frac{\log N}{4\sqrt{N}},$ we have $$\nabla V(s) \leq -\frac{\mu_1\mu_2}{4}\frac{\log N}{\sqrt{N}}.$$
\end{lemma}

\begin{proof}
When $V(s) = s_{1,2} - \frac{p}{\mu_2} \geq \frac{\log N}{4\sqrt{N}},$ we have 
\begin{align}
\nabla V(s) =& p\mu_1s_{1,1} - \mu_2s_{1,2} \label{s12-up-1}\\ 
       \leq& p\mu_1 - (p\mu_1+\mu_2)s_{1,2} \label{s12-up-2}\\
          =& \mu_1(p - \mu_2s_{1,2}) \leq -\frac{\mu_1\mu_2}{4}\frac{\log N}{\sqrt{N}} \label{s12-up-3} 
\end{align}
\eqref{s12-up-1} to \eqref{s12-up-2} holds because $s_{1,1} = s_{1} - s_{1,2} \leq 1 - s_{1,2}$ (note this structure is simple yet critical in proving Lemma \ref{driftbound-u:s12} and driving iterative SSC in Figure \ref{fig:roadmap}); \eqref{s12-up-2} to \eqref{s12-up-3} holds because $\frac{1}{\mu_1} + \frac{p}{\mu_2} = 1$ implies $p\mu_1+\mu_2=\mu_1\mu_2.$
\end{proof}

From Lemma \ref{driftbound-u:s12}, we know $B = \frac{\log N}{4\sqrt{N}}$ and $\gamma = \frac{\mu_1\mu_2}{4}\frac{\log N}{\sqrt{N}}.$  According to the definition of $q_{\max}$ and $\nu_{\max},$ we have $q_{\max} = \mu_{\max}N$ and $\nu_{\max} = \frac{1}{N}.$ Since $\mathcal E= \mathcal S^{(N)}$ is the entire space, then $\mathbb P\left(S \notin \mathcal E\right)=0,$ we use Lemma \ref{tail-bound-cond} (or Theorem 1 in \cite{BerGamTsi_01}) to obtain the following tail bound with $j=\frac{\sqrt{N}\log N}{8},$
\begin{align}
\mathbb P\left(V(S) \geq B + 2 \nu_{\max}j\right) 
=& \mathbb P\left(S_{1,2} - \frac{p}{\mu_2} \geq \frac{\log N}{2\sqrt{N}}\right) \label{s12-tail-1}\\
\leq& \left(\frac{1}{1 + \frac{\mu_1\mu_2}{4\mu_{\max}}\frac{\log N}{\sqrt{N}}}\right)^{\frac{\sqrt{N}\log N}{8}} \label{s12-tail-2}\\
\leq& \left(1 - \frac{\mu_1\mu_2}{5\mu_{\max}}\frac{\log N}{\sqrt{N}}\right)^{\frac{\sqrt{N}\log N}{8}} \label{s12-tail-3}\\
\leq& e^{-\frac{\mu_1\mu_2\log^2 N}{40\mu_{\max}}} \label{s12-tail-4}                             
\end{align}
\begin{itemize}
\item \eqref{s12-tail-1} holds by substituting $B=\frac{\log N}{4\sqrt{N}},$ $\nu_{\max}=\frac{1}{N}$ and $j=\frac{\sqrt{N}\log N}{8};$
\item \eqref{s12-tail-1} to \eqref{s12-tail-2} holds based on Lemma \ref{driftbound-u:s12};
\item \eqref{s12-tail-2} to \eqref{s12-tail-3} holds because $\frac{\mu_1\mu_2}{\mu_{\max}} \leq \frac{\sqrt{N}}{\log N}$ for a large $N$ satisfying \eqref{N-cond}. 
\end{itemize}

\subsection{Proof of Lemma \ref{tailbound-l:s11}: A lower bound on $S_{1,1}.$}

To prove Lemma \ref{tailbound-l:s11}, we first establish a Lyaponuv drift analysis in Lemma \ref{driftbound-l:s11}.

\begin{lemma}\label{driftbound-l:s11}
Consider Lyapunov function 
\begin{align}
V(s) = \frac{\lambda}{\mu_1} - s_{1,1}. \label{Ly:s11}
\end{align} We have 
\begin{itemize}
\item $\nabla V(s) \leq - \frac{\mu_1}{3}\frac{\log N}{\sqrt{N}},$
when $$V(s) \geq \frac{\log N}{2 \sqrt{N}} ~~\text{and}~~ s_{1,2} \leq \frac{p}{\mu_2}+\frac{\log N}{2\sqrt{N}};
$$
\item $\nabla V(s) \leq 1,$
when $$V(s) \geq \frac{\log N}{2 \sqrt{N}} ~~\text{and}~~ s_{1,2} \geq \frac{p}{\mu_2}+\frac{\log N}{2\sqrt{N}}.$$
\end{itemize}
\end{lemma}

\begin{proof}
Assuming $s_{1,2} \leq \frac{p}{\mu_2}+\frac{\log N}{2\sqrt{N}}$ and $\frac{\lambda}{\mu_1} - s_{1,1} \geq \frac{\log N}{2 \sqrt{N}},$ 
we have $$s_1 = s_{11} + s_{12} \leq \frac{p}{\mu_2} + \frac{\lambda}{\mu_1} = 1 - \frac{1}{\mu_1 N^\alpha} \leq \lambda + \frac{1+\mu_1+\mu_2}{w_l}\frac{\log N}{\sqrt{N}}<1.$$
Therefore, the drift of $V(s)$ is
\begin{align} 
\nabla V(s) =& -\lambda {(1-A_1(s))} + \mu_1s_{1,1} - (1-p)\mu_1s_{2,1} - \mu_2 s_{2,2} \label{s11-lower-1}\\
         \leq&  {\frac{1}{\sqrt{N}}}-\lambda + \mu_1s_{1,1} - (1-p)\mu_1s_{2,1} - \mu_2 s_{2,2} \label{s11-lower-2}\\
         \leq&  {\frac{1}{\sqrt{N}}}-\lambda  + \mu_1s_{1,1} \label{s11-lower-3}\\
         \leq&  {\frac{1}{\sqrt{N}}}-\frac{\mu_1}{2}\frac{\log N}{\sqrt{N}} \label{s11-lower-4}\\
         \leq&  -\frac{\mu_1}{3}\frac{\log N}{\sqrt{N}} \label{s11-lower-5},
\end{align} where
\begin{itemize}
\item \eqref{s11-lower-1} to \eqref{s11-lower-2} holds because {$A_1(s) \leq \frac{1}{\sqrt{N}}$} under any policy in $\Pi$;
\item \eqref{s11-lower-3} to \eqref{s11-lower-4} holds because $s_{1,1} \leq \frac{\lambda}{\mu_1}-\frac{\log N}{2 \sqrt{N}}.$
\end{itemize}

Assuming $s_{12} > \frac{p}{\mu_2}+\frac{\log N}{2\sqrt{N}}$ and $s_{1,1} \leq \frac{\lambda}{\mu_1}-\frac{\log N}{2 \sqrt{N}},$ we have 
$$\nabla V(s) = -\lambda {(1-A_1(s))} + \mu_1s_{1,1} - (1-p)\mu_1s_{2,1} - \mu_2 s_{2,2} \leq  \mu_1 s_{1,1} < 1.$$ 
\end{proof}

Let $\mathcal E = \left\{s ~|~ s \leq \frac{p}{\mu_2}+\frac{\log N}{2\sqrt{N}}\right\}.$  we have $V(s) = \frac{\lambda}{\mu_1} - s_{1,1}$ satisfying two conditions:
\begin{itemize}
\item $\nabla V(s) \leq -\frac{\mu_1}{3}\frac{\log N}{\sqrt{N}}$ when $V(s) \geq \frac{\log N}{2 \sqrt{N}}$ and $s_{1,2} \in \mathcal E.$
\item $\nabla V(s) \leq 1$ when $V(s) \geq  \frac{\log N}{2 \sqrt{N}}$ and $s_{1,2} \notin \mathcal E.$
\end{itemize}
Define $B = \frac{\log N}{2 \sqrt{N}}$, $\gamma = \frac{\mu_1}{3}\frac{\log N}{\sqrt{N}},$ and $\delta = 1.$ Combining $q_{\max} \leq \mu_{\max}N$ and $\nu_{\max} \leq \frac{1}{N},$ we have $$\alpha \leq \frac{1}{1 + \frac{\mu_1}{3\mu_{\max}}\frac{\log N}{\sqrt{N}}} ~~\text{and}~~ \beta = \frac{1}{\frac{\mu_1}{3}\frac{\log N}{\sqrt{N}}}+1.$$ 

Based on Lemma \ref{tail-bound-cond} with $j=\frac{\sqrt{N}\log N}{4},$ we have
\begin{align}
\mathbb P\left(V(S) \geq B + 2 \nu_{\max} j\right) 
=& \mathbb P\left(\frac{\lambda}{\mu_1} - S_{1,1} \geq  \frac{\log N}{\sqrt{N}}\right) \label{s11-tail-1}\\
\leq& \left(\frac{1}{1 + \frac{\mu_1}{3\mu_{\max}}\frac{\log N}{\sqrt{N}}}\right)^{\frac{\sqrt{N}\log N}{4}} + \beta \mathbb P\left(S_{1,2} \notin {\mathcal E}\right)\label{s11-tail-2}\\
\leq& \left(1 - \frac{\mu_1}{4\mu_{\max}}\frac{\log N}{\sqrt{N}}\right)^{\frac{\sqrt{N}\log N}{4}} + \frac{4}{\mu_1}\frac{\sqrt{N}}{\log N} e^{-\frac{\mu_1\mu_2\log^2 N}{40\mu_{\max}}} \label{s11-tail-3}\\
\leq& e^{-\frac{\mu_1\log^2 N}{16\mu_{\max}}} + \frac{4}{\mu_1}\frac{\sqrt{N}}{\log N} e^{-\frac{\mu_1\mu_2\log^2 N}{40\mu_{\max}}} \label{s11-tail-4} \\
\leq& \frac{5}{\mu_1}\frac{\sqrt{N}}{\log N} e^{-\min\left(\frac{\mu_1}{16\mu_{\max}},\frac{\mu_1\mu_2}{40\mu_{\max}}\right)\log^2 N} \label{s11-tail-5},
\end{align} where
\begin{itemize}
\item \eqref{s11-tail-1} holds by substituting $B=\frac{\log N}{2\sqrt{N}},$ $\nu_{\max}=\frac{1}{N}$ and $j=\frac{\sqrt{N}\log N}{4};$
\item \eqref{s11-tail-1} to \eqref{s11-tail-2} holds based on Lemma \ref{driftbound-l:s11};
\item \eqref{s11-tail-2} to \eqref{s11-tail-3} holds because (i) in the first term in \eqref{s11-tail-2}, $\frac{\mu_1}{\mu_{\max}} \leq \frac{\sqrt{N}}{\log N}$ for a large $N$ satisfying \eqref{N-cond}, and (ii) the second term in \eqref{s11-tail-2} can be bounded by applying Lemma \ref{tailbound-u:s12}.
\end{itemize}

\subsection{Proof of Lemma \ref{tailbound-l:s12}: A lower bound on $S_{1,2}.$}

\begin{lemma}\label{driftbound-l:s12}
Consider Lyapunov function $$V(s) = \frac{p\lambda}{\mu_2}-s_{1,2}.$$ We have 
\begin{itemize}
\item $\nabla V(s) \leq -\frac{\mu_2}{2}\frac{\log N}{\sqrt{N}},$ when $$V(s) \geq \left(\frac{p \mu_1}{\mu_2}+\frac{1}{2}\right) \frac{\log N}{\sqrt{N}} ~~\text{and}~~ s_{1,1} \geq \frac{\lambda}{\mu_1} - \frac{\log N}{\sqrt{N}};$$
\item $\nabla V(s) \leq 1,$ when $$V(s) \geq \left(\frac{p \mu_1}{\mu_2}+\frac{1}{2}\right) \frac{\log N}{\sqrt{N}} ~~\text{and}~~ s_{1,1} \leq \frac{\lambda}{\mu_1} - \frac{\log N}{\sqrt{N}}.$$
\end{itemize}
\end{lemma}

\begin{proof}
Assuming $V(s) = \frac{p\lambda}{\mu_2}-s_{1,2} \geq \left(\frac{p \mu_1}{\mu_2}+\frac{1}{2}\right)\frac{\log N}{\sqrt{N}}$ and $s_{1,1} \geq \frac{\lambda}{\mu_1} - \frac{\log N}{\sqrt{N}},$ we have 
\begin{align}
\nabla V(s) =& -(p\mu_1s_{1,1} - \mu_2s_{1,2}) \label{s12-lower-1} \\
         \leq& -\left(p\lambda - \frac{p\mu_1\log N}{\sqrt{N}} - \mu_2s_{1,2}\right) \label{s12-lower-2} \\
         \leq& - \frac{\mu_2}{2} \frac{\log N}{\sqrt{N}}, \label{s12-lower-3}
\end{align} where
\begin{itemize}
\item \eqref{s12-lower-1} to \eqref{s12-lower-2} holds because $s_{1,1} \geq \frac{\lambda}{\mu_1} - \frac{\log N}{\sqrt{N}};$ 
\item \eqref{s12-lower-2} to \eqref{s12-lower-3} holds because $s_{1,2} \leq \frac{p\lambda}{\mu_2} - \left(\frac{p \mu_1}{\mu_2}+\frac{1}{2}\right)\frac{\log N}{\sqrt{N}}.$
\end{itemize}

Next, assuming $\frac{p\lambda}{\mu_2}-s_{1,2} \geq \left(\frac{p \mu_1}{\mu_2}+\frac{1}{2}\right)\frac{\log N}{\sqrt{N}}$ and $s_{1,1} < \frac{\lambda}{\mu_1} - \frac{\log N}{\sqrt{N}},$ we have 
\begin{align}
\nabla V(s) = -(p\mu_1s_{1,1} - \mu_2s_{1,2}) \leq \mu_2s_{1,2} \leq p \lambda \leq 1. 
\end{align}
\end{proof}

Defining $\mathcal E = \left\{s ~|~ s \geq \frac{\lambda}{\mu_1} - \frac{\log N}{\sqrt{N}}\right\},$ we have $V(s) = \frac{p\lambda}{\mu_2} -s_{1,2}$ satisfying two conditions:
\begin{itemize}
\item $\nabla V(s) \leq -\frac{\mu_2}{2}\frac{\log N}{\sqrt{N}}$ when $V(s) \geq \left(\frac{p \mu_1}{\mu_2}+\frac{1}{2}\right) \frac{\log N}{\sqrt{N}}$ and $s_{1,1} \in \mathcal E.$
\item $\nabla V(s) \leq 1$ when $V(s) \geq \left(\frac{p \mu_1}{\mu_2}+\frac{1}{2}\right) \frac{\log N}{\sqrt{N}}$ and $s_{1,1} \notin \mathcal E.$
\end{itemize}
Define $B = \left(\frac{p \mu_1}{\mu_2}+\frac{1}{2}\right) \frac{\log N}{\sqrt{N}}$, $\gamma = \frac{\mu_2}{2}\frac{\log N}{\sqrt{N}}$ and $\delta = 1.$ Combining $q_{\max} = \mu_{\max}N$ and $\nu_{\max} = \frac{1}{N},$ we have $$\alpha \leq \frac{1}{1 + \frac{\mu_2}{2\mu_{\max}}\frac{\log N}{\sqrt{N}}} ~~\text{and}~~ \beta = \frac{2}{\mu_2}\frac{\sqrt{N}}{\log N} + 1.$$ 

Based on Lemma \ref{tail-bound-cond} with $j= \frac{\sqrt{N}\log N}{4}$, we have
\begin{align}
\mathbb P\left(V(S) \geq B + 2 \nu_{\max} j\right) 
=& \mathbb P\left( \frac{p\lambda}{\mu_2} -S_{1,2} \geq \left(\frac{p \mu_1}{\mu_2}+1\right) \frac{\log N}{\sqrt{N}}\right) \label{s12-lower-tail-1}\\
\leq& \left(\frac{1}{1 + \frac{\mu_2}{2\mu_{\max}}\frac{\log N}{\sqrt{N}}} \right)^{\frac{\sqrt{N}\log N}{4}} + \frac{2}{\mu_2}\frac{\sqrt{N}}{\log N} \mathbb P\left(S_{1,1} \notin \mathcal E \right)\label{s12-lower-tail-2}\\
\leq& \left(1 - \frac{\mu_2}{3\mu_{\max}}\frac{\log N}{\sqrt{N}}\right)^{\frac{\sqrt{N}\log N}{4}} + \frac{3}{\mu_2}\frac{\sqrt{N}}{\log N} \mathbb P\left(S_{1,1} \notin \mathcal E \right)\label{s12-lower-tail-3}\\                          
\leq& e^{-\frac{\mu_2\log^2 N}{12\mu_{\max}}} + \frac{15}{\mu_1\mu_2} \frac{N}{\log^2 N} e^{-\min\left(\frac{\mu_1}{16\mu_{\max}},\frac{\mu_1\mu_2}{40\mu_{\max}}\right)\log^2 N} \label{s12-lower-tail-4}\\
\leq& \frac{16}{\mu_1\mu_2} \frac{N}{\log^2 N} e^{-\min\left(\frac{\mu_1}{16\mu_{\max}},\frac{\mu_2}{12\mu_{\max}},\frac{\mu_1\mu_2}{40\mu_{\max}}\right)\log^2 N},  \label{s12-lower-tail-5}              
\end{align} where
\begin{itemize}
\item \eqref{s12-lower-tail-1} holds by substituting $B,$ $\nu_{\max}$ and $j;$
\item \eqref{s12-lower-tail-1} to \eqref{s12-lower-tail-2} holds due to Lemma \ref{driftbound-l:s12};
\item \eqref{s12-lower-tail-2} to \eqref{s12-lower-tail-3} holds because $\frac{\mu_2}{\mu_{\max}} \leq \frac{\sqrt{N}}{\log N}$ for $N$ satisfying \eqref{N-cond} in the first term of \eqref{s12-lower-tail-3};
\item \eqref{s12-lower-tail-3} to \eqref{s12-lower-tail-4} holds by Lemma \ref{tailbound-l:s11} to obtain the tail bound in the second term of \eqref{s12-lower-tail-4}.
\end{itemize}

Recall $\frac{p\mu_1}{\mu_2}+1= \mu_1$ and the proof is completed.


\subsection{Proof of Lemma \ref{SSC:s1 s2}: SSC on $S_1$ and $\sum_{i=2}^{b}S_i.$}
Define $L_{1,1} = \frac{\lambda}{\mu_1} - \frac{\log N}{\sqrt{N}}$ and $L_{1,2} = \frac{p\lambda}{\mu_2}- \frac{\mu_1\log N}{\sqrt{N}}.$ Recall $w_u = \max((1-p)\mu_1, \mu_2),$ $w_l = \min((1-p)\mu_1, \mu_2),$ $k = \left(1+\frac{w_u b}{w_l}\right)\left(\frac{1+\mu_1+\mu_2}{w_l}+2\mu_1\right)$ and $c_1 = \frac{w_u b}{w_l}\left(\frac{1+\mu_1+\mu_2}{w_l}+2\mu_1\right)+2\mu_1.$

\begin{lemma}\label{driftbound:SSCp=0:s1 s2}
Consider Lyapunov function 
\begin{align}
V(s) = \min\left\{\lambda +\frac{k\log N}{\sqrt{N}}-s_1, \sum_{i=2}^b s_i\right\}. \label{Ly:SSC}
\end{align} 
We have
\begin{itemize}
\item $\nabla V(s) \leq -\frac{w_u\mu_1\log N}{\sqrt{N}},$ when $$V(s) \geq \frac{c_1\log N}{\sqrt{N}} ~\text{with} ~s_{1,1} \geq L_{1,1} ~\text{and}~ s_{1,2} \geq L_{1,2};$$
\item $\nabla V(s) \leq w_{u},$ when $$V(s) \geq \frac{c_1\log N}{\sqrt{N}} ~\text{with} ~s_{1,1} \leq L_{1,1} ~\text{or}~ s_{1,2} \leq L_{1,2}.$$
\end{itemize}
\end{lemma} 
\begin{proof}
When $V(s) \geq \frac{c_1\log N}{\sqrt{N}},$ the following two inequalities hold
\begin{align}
             &s_1 \leq \lambda + \frac{(k-c_1)\log N}{\sqrt{N}} = \lambda + \frac{1+\mu_1+\mu_2}{w_l}\frac{\log N}{\sqrt{N}}, \label{SSCp=0:s1}\\
&\sum_{i=2}^b s_i \geq \frac{c_1\log N}{\sqrt{N}}. \label{SSCp=0:s2}
\end{align}

We have two observations based on \eqref{SSCp=0:s1} and \eqref{SSCp=0:s2}:
\begin{itemize}
\item \eqref{SSCp=0:s1} implies $A_1(s) \leq \frac{1}{\sqrt{N}}$ under any policy in $\Pi$;  
\item \eqref{SSCp=0:s2} implies $s_2 \geq \frac{c_1}{b}\frac{\log N}{\sqrt{N}}$ because  $s_2 \geq s_3 \geq \cdots \geq s_b,$ and we have
\end{itemize}
\begin{align}
(1-p)\mu_1 s_{2,1} + \mu_2 s_{2,2} \geq w_l s_2 \geq \frac{w_lc_1}{b}\frac{\log N}{\sqrt{N}}, \label{p=0s21-s22-lowerbound}
\end{align}
where a finite buffer size is required such that the lower bound $w_l s_2 \geq \frac{w_lc_1}{b}\frac{\log N}{\sqrt{N}}$ is meaningful.

We study the Lyapunov drift and consider two cases:
\begin{itemize}
\item Supppose $\lambda +\frac{k\log N}{\sqrt{N}}-s_1 \geq \sum_{i=2}^b s_i \geq \frac{c_1\log N}{\sqrt{N}}.$ In this case, $V(s) = \sum_{i=2}^b s_i,$ and
\begin{align}
\nabla V(s) \leq& \lambda {(A_1(s)-A_b(s))}- (1- p)\mu_1 s_{2,1} -\mu_2 s_{2,2} \label{SSCp=0-case1-1}\\
            \leq& {\frac{1}{\sqrt{N}}}- (1- p)\mu_1 s_{2,1} -\mu_2 s_{2,2} \label{SSCp=0-case1-2}\\
            \leq&  {\frac{1}{\sqrt{N}}}-\frac{w_lc_1}{b}\frac{\log N}{\sqrt{N}} \label{SSCp=0-case1-3}\\
            \leq&  {\frac{1}{\sqrt{N}}}-\frac{2w_u \mu_1\log N}{\sqrt{N}} \label{SSCp=0-case1-4}\\
            \leq&  -\frac{w_u \mu_1\log N}{\sqrt{N}},
\end{align}  where
\begin{itemize}
\item \eqref{SSCp=0-case1-1} to \eqref{SSCp=0-case1-2} holds because {$A_1(s) \leq \frac{1}{\sqrt{N}}$ under any policy in $\Pi;$}
\item \eqref{SSCp=0-case1-2} to \eqref{SSCp=0-case1-3} holds because \eqref{p=0s21-s22-lowerbound};
\item \eqref{SSCp=0-case1-3} to \eqref{SSCp=0-case1-4} holds because $c_1 \geq \frac{w_u b}{w_l} 2\mu_1.$
\end{itemize}

\item Suppose $\sum_{i=2}^b s_i > \lambda +\frac{k\log N}{\sqrt{N}}-s_1 \geq \frac{c_1\log N}{\sqrt{N}}.$ In this case, $V(s) = \lambda +\frac{k\log N}{\sqrt{N}}-s_1,$ and
\begin{align}
&\nabla V(s) \\
\leq& -\lambda {(1-A_1(s))} + (1-p)\mu_1 s_{1,1} + \mu_2 s_{1,2} - (1- p)\mu_1 s_{2,1} - \mu_2 s_{2,2} \label{SSCp=0-case2-1}\\ 
            \leq& {\frac{1}{\sqrt{N}}}-\lambda + w_u s_1 - \left(w_u-(1-p)\mu_1\right) s_{1,1} - \left(w_u-\mu_2\right) s_{1,2} - (1- p)\mu_1 s_{2,1} - \mu_2 s_{2,2} \label{SSCp=0-case2-2}\\ 
       {\leq}& {\frac{1}{\sqrt{N}}} -\lambda + w_u (s_1-L_{1,1}-L_{1,2}) + \left((1-p)\mu_1L_{1,1} + \mu_2L_{1,2}\right)- (1- p)\mu_1 s_{2,1} - \mu_2 s_{2,2} \label{SSCp=0-case2-3}\\
       {=}& {\frac{1}{\sqrt{N}}} + \left(w_u (k-c_1 + 1 + \mu_1) - (1-p)\mu_1 - \mu_1\mu_2 \right)\frac{\log N}{\sqrt{N}}- (1- p)\mu_1 s_{2,1} - \mu_2 s_{2,2} \label{SSCp=0-case2-4}\\
       {\leq}& {\frac{1}{\sqrt{N}}} + \left(w_u (k-c_1 + 1 + \mu_1) - (1-p)\mu_1 - \mu_1\mu_2 \right)\frac{\log N}{\sqrt{N}} - \frac{w_lc_1}{b}\frac{\log N}{\sqrt{N}} \label{SSCp=0-case2-5}\\
       {=}& w_u\left( k-\left(1+\frac{w_l}{w_ub}\right)c_1 + \mu_1\right)\frac{\log N}{\sqrt{N}} +{\frac{1}{\sqrt{N}}}- \left((1-p)\mu_1 + \mu_1\mu_2 - w_u\right)\frac{\log N}{\sqrt{N}} \label{SSCp=0-case2-6}\\
       {\leq}& w_u\left( k-\left(1+\frac{w_l}{w_ub}\right)c_1 + \mu_1\right)\frac{\log N}{\sqrt{N}} \label{SSCp=0-case2-7}\\
       {\leq}& -\frac{w_u \mu_1\log N}{\sqrt{N}}, \label{SSCp=0-case2-8}
\end{align} where
\begin{itemize}
\item \eqref{SSCp=0-case2-1} to \eqref{SSCp=0-case2-2} holds by adding and substructing $w_us_1 = w_u(s_{1,1}+s_{1,2});$
\item \eqref{SSCp=0-case2-2} to \eqref{SSCp=0-case2-3} holds because $s_{1,1}$ and $s_{1,2}$ taking the lower bounds at $L_{1,1}$ and $L_{1,2}$ gives an upper bound;
\item \eqref{SSCp=0-case2-3} to \eqref{SSCp=0-case2-4} holds by substituting $L_{1,1} = \frac{\lambda}{\mu_1} - \frac{\log N}{\sqrt{N}},$  $L_{1,2} = \frac{p\lambda}{\mu_2}- \frac{\mu_1\log N}{\sqrt{N}}$ and $s_1 \leq \lambda + \frac{(k-c_1)\log N}{\sqrt{N}}.$ We have $s_1-L_{1,1}-L_{1,2} = (k-c_1+1 + \mu_1)\frac{\log N}{\sqrt{N}}$ and $(1-p)\mu_1L_{1,1} + \mu_2L_{1,2} = \lambda - \left((1-p)\mu_1 + \mu_1\mu_2\right)\frac{\log N}{\sqrt{N}}.$
\item \eqref{SSCp=0-case2-4} to \eqref{SSCp=0-case2-5} holds by substituting the lower bound of $(1- p)\mu_1 s_{2,1} + \mu_2 s_{2,2}$ in \eqref{p=0s21-s22-lowerbound};
\item \eqref{SSCp=0-case2-5} to \eqref{SSCp=0-case2-6} holds by combining the terms with $c_1;$ 
\item \eqref{SSCp=0-case2-6} to \eqref{SSCp=0-case2-7} holds because {$((1-p)\mu_1 + \mu_1\mu_2 - w_u)\log N = (\mu_1 + \mu_2 - w_u)\log N \geq 1$ when $N$ satisfies \eqref{N-cond};}
\item \eqref{SSCp=0-case2-7} to \eqref{SSCp=0-case2-8} holds because $k -\left(1+\frac{w_l}{w_ub}\right)c_1 \leq - 2\mu_1.$
\end{itemize}
\end{itemize}

Next, we show $\nabla V(s) \leq w_{u}$ based on the upper bounds \eqref{SSCp=0-case1-1} and \eqref{SSCp=0-case2-1}.
\begin{itemize}
\item Consider the upper bound in \eqref{SSCp=0-case1-1}. We have
\begin{align}
\nabla V(s) \leq \lambda (A_1(s)-A_b(s)) - (1- p)\mu_1 s_{2,1} -\mu_2 s_{2,2} \leq 1 \leq w_{u}, \nonumber
\end{align}
where $1 \leq w_{u}$ holds because $\frac{1}{\mu_1} + \frac{p}{\mu_2}=1.$
\item Consider the upper bound in \eqref{SSCp=0-case2-1}. We have
\begin{align}
\nabla V(s) \leq& -\lambda (1-A_1(s)) + (1-p)\mu_1 s_{1,1} + \mu_2 s_{1,2} - (1- p)\mu_1 s_{2,1} - \mu_2 s_{2,2} \nonumber\\
\leq& (1-p)\mu_1 s_{1,1} + \mu_2 s_{1,2} \leq w_{u}, \nonumber
\end{align}
where the last inequality holds because $s_{1,1}+s_{1,2} = s_1 \leq 1.$
\end{itemize}

\end{proof}

Let $\mathcal E = \left\{s ~|~ s_{1,1} \geq L_{1,1}, ~s_{1,2} \geq L_{1,2}\right\}.$ We have $V(s) = \min\left\{\lambda +\frac{k\log N}{\sqrt{N}}-s_1, \sum_{i=2}^{b}s_i\right\}$ satisfying the following two conditions based on Lemma \ref{driftbound:SSCp=0:s1 s2}:
\begin{itemize}
\item $\nabla V(s) \leq -\frac{w_u\mu_1\log N}{\sqrt{N}}$ when $V(s) \geq \frac{c_1\log N}{\sqrt{N}}$ and $s \in \mathcal E.$
\item $\nabla V(s) \leq w_u$ when $V(s) \geq \frac{c_1\log N}{\sqrt{N}}$ and $s \notin \mathcal E.$
\end{itemize}
Define  $B = \frac{c_1\log N}{\sqrt{N}},$ $\gamma = \frac{w_u\mu_1\log N}{\sqrt{N}}$ and $\delta = w_u.$ Combining $q_{\max} = \mu_{\max}N$ and $\nu_{\max} = \frac{1}{N},$ we have $$\alpha \leq \frac{1}{1 + \frac{w_u\mu_1\log N}{\mu_{\max}\sqrt{N}}} ~~\text{and}~~ \beta = \frac{\sqrt{N}}{\mu_1\log N} + 1.$$ 

Based on Lemma \ref{tail-bound-cond} with $j=\frac{\mu_1\sqrt{N}\log N}{2},$ we have
\begin{align}
&\mathbb P\left(V(S) \geq B + 2 \nu_{\max} j\right) \\
=& \mathbb P\left(V(S) \geq \frac{c_1\log N}{\sqrt{N}} + \frac{\mu_1\log N}{\sqrt{N}} \right) \label{SSC1p=0-tail-1}\\
\leq& \left(\frac{1}{1 + \frac{w_u\mu_1\log N}{\mu_{\max}\sqrt{N}}}\right)^{\frac{\mu_1\sqrt{N}\log N}{2}} + \left(\frac{\sqrt{N}}{\mu_1\log N} + 1\right) \mathbb P\left(s \notin \mathcal E\right) \label{SSC1p=0-tail-2}\\
\leq& \left(1 - \frac{w_u\mu_1}{2\mu_{\max}}\frac{\log N}{\sqrt{N}}\right)^{\frac{\mu_1\sqrt{N}\log N}{2}} + \left(\frac{\sqrt{N}}{\mu_1\log N} + 1\right) \mathbb P\left(s \notin \mathcal E\right) \label{SSC1p=0-tail-3}\\
\leq& e^{-\frac{w_u\mu_1^2\log^2 N}{4\mu_{\max}}} + \left(\frac{\sqrt{N}}{\mu_1\log N} + 1\right)\frac{32}{\mu_1\mu_2} \frac{N}{\log^2 N} e^{-\min\left(\frac{\mu_1}{16\mu_{\max}},\frac{\mu_2}{12\mu_{\max}},\frac{\mu_1\mu_2}{40\mu_{\max}}\right)\log^2 N} \label{SSC1p=0-tail-4}\\
\leq& \frac{34}{\mu_1^2\mu_2} \frac{N^{1.5}}{\log^3 N} e^{-\min\left(\frac{\mu_1}{16\mu_{\max}},\frac{\mu_2}{12\mu_{\max}},\frac{\mu_1\mu_2}{40\mu_{\max}}\right)\log^2 N}, \label{SSC1p=0-tail-5}
\end{align}  where
\begin{itemize}
\item \eqref{SSC1p=0-tail-1} holds holds by substituting $B,$ $\nu_{\max}$ and $j;$
\item \eqref{SSC1p=0-tail-1} to \eqref{SSC1p=0-tail-2} holds based on Lemma \ref{driftbound:SSCp=0:s1 s2};
\item \eqref{SSC1p=0-tail-2} to \eqref{SSC1p=0-tail-3} holds $\frac{w_u\mu_1}{\mu_{\max}} \leq \frac{\sqrt{N}}{\log N}$ for a large $N$ for the first term in \eqref{SSC1p=0-tail-3};
\item \eqref{SSC1p=0-tail-3} to \eqref{SSC1p=0-tail-4} holds by applying the union bound on $\mathbb P\left(S \notin \mathcal E\right)$ such that  
\begin{align*}
\mathbb P\left(s \notin \mathcal E\right) \leq& \mathbb P\left(s_{1,1} < L_{1,1}\right)+ \mathbb P\left(s_{1,2} < L_{1,2}\right) \\
\leq& \frac{32}{\mu_1\mu_2} \frac{N}{\log^2 N} e^{-\min\left(\frac{\mu_1}{16\mu_{\max}},\frac{\mu_2}{12\mu_{\max}},\frac{\mu_1\mu_2}{40\mu_{\max}}\right)\log^2 N}.
\end{align*}
\end{itemize}

\section{Proof of the Corollary} \label{sec:0-delay}
Under JSQ, a job is discarded or blocked only if all buffers are full, i.e. when $N\sum_{i=1}^b S_i= Nb.$  From Theorem \ref{Thm:main}, we have
\begin{align}
\mathbb P(\mathcal B) = &\mathbb P\left(N\sum_{i=1}^b S_i=Nb\right)=\mathbb P\left(\sum_{i=1}^b S_i\geq b\right)\\
\leq &\mathbb P\left(\max\left\{\sum_{i=1}^b S_i-\lambda -\frac{k\log N}{\sqrt{N}}, 0\right\}\geq b-\lambda -\frac{k\log N}{\sqrt{N}}\right) \label{markov-1}\\
{\leq}&\frac{\mathbb E\left[\max\left\{\sum_{i=1}^b S_i-\lambda -\frac{k\log N}{\sqrt{N}}, 0\right\}\right]}{b-\lambda - \frac{k\log N}{\sqrt{N}}}\label{markov-2}\\
{\leq}&{\frac{8\mu_{\max}}{b-\lambda}\frac{1}{\sqrt{N}\log N}} \label{markov-3}
\end{align}
where \eqref{markov-1} to \eqref{markov-2} holds due to the Markov inequality; and \eqref{markov-2} to \eqref{markov-3} holds because of Theorem \ref{Thm:main} and $b-\lambda \geq \frac{8k\log N}{\sqrt{N}}$.

For jobs that are not discarded, the average queueing delay according to Little's law is
$$\frac{\mathbb E\left[\sum_{i=1}^bS_i\right]}{\lambda(1-\mathbb P(\mathcal B))}.$$ Therefore, the average waiting time is
\begin{align*}
\mathbb E[W]=&\frac{\mathbb E\left[\sum_{i=1}^bS_i\right]}{\lambda(1-\mathbb P(\mathcal B)))}-1 \\
\leq& \frac{\frac{k\log N}{\sqrt{N}}+\frac{7\mu_{\max}}{\sqrt{N}\log N}+\lambda \mathbb P(\mathcal B)}{\lambda(1-\mathbb P(\mathcal B))} \\
\leq& \frac{2k\log N}{\sqrt{N}} + \frac{14\mu_{\max}+\frac{16\mu_{\max}}{b-\lambda}}{\sqrt{N}\log N},
\end{align*}
where the last inequality holds because $\lambda(1-\mathbb P(\mathcal B)) \geq 0.5$ under $b-\lambda \geq \frac{8k\log N}{\sqrt{N}}.$

Next, we study the waiting probability $\mathbb P(\mathcal W)$. Define $\overline{\mathcal W}$ to be the event that a job entered into the system (not blocked) and waited in the buffer and $\mathbb P(\overline{\mathcal W})$ is the steady-state probability of $\overline{\mathcal W}$. {Applying Little's law to the jobs waiting in the buffer, $$\lambda \mathbb P(\overline{\mathcal W}) \mathbb E[T_{Q}] = \mathbb E\left[\sum_{i=2}^b S_i\right],$$ where $T_{Q}$ is the waiting time for the jobs waiting in the buffer.}
Since $\mathbb E[T_{Q}]$ is lower bounded by $\overline T_{Q} = \min \left\{\frac{1}{\mu_1}, \frac{1}{\mu_2}\right\},$ we have
$$\mathbb P(\overline{\mathcal W}) \leq \frac{\mathbb E\left[\sum_{i=2}^b S_i\right]}{\lambda \overline T_{Q}}.$$

We now provide a bound on $\mathbb E\left[\sum_{i=2}^b S_i\right]$.
From the work-conserving law, we have 
$$\mathbb E[S_1]=\lambda(1-\mathbb P(\mathcal B))\geq \lambda\left(1-\frac{8\mu_{\max}}{b-\lambda}\frac{1}{\sqrt{N}\log N}\right).$$ Therefore, we have
$$\mathbb E[S_1]\ge \lambda - \frac{8\mu_{\max}}{b-\lambda}\frac{1}{\sqrt{N}\log N}.$$ From Theorem \ref{Thm:main}, one has
$$\mathbb E\left[\sum_{i=1}^b {S}_i\right]\leq \lambda + \frac{k\log N}{\sqrt{N}}+\frac{7\mu_{\max}}{\sqrt{N}\log N}.$$ The above two inequalities give the following bound on $\mathbb E\left[\sum_{i=2}^b S_i\right]$:  $$\mathbb E\left[\sum_{i=2}^b S_i\right] \leq \frac{k\log N}{\sqrt{N}} + \frac{7\mu_{\max}+\frac{8\mu_{\max}}{b-\lambda}}{\sqrt{N}\log N}.$$

Finally, a job not routed to an idle server is either blocked or waited in the buffer
{$$\mathbb P(\mathcal W)=  \mathbb P(\mathcal B)+\mathbb P(\overline{\mathcal W}) \leq \mathbb P(\mathcal B)+\frac{\mathbb E\left[\sum_{i=2}^b S_i\right]}{\lambda \overline T_{Q}} \leq \frac{1}{\lambda \overline T_{Q}}\frac{k\log N}{\sqrt{N}} + \frac{1}{\lambda\overline T_{Q}} \frac{7\mu_{\max}+\frac{8\mu_{\max}}{b-\lambda}}{\sqrt{N}\log N}.$$}

The analysis for Po$d$ is similar, except that
\begin{align}
\mathbb P(\mathcal B)=&\mathbb P\left(\mathcal B\left|S_b\leq 1-\frac{1}{\mu_1 N^{\alpha}}\right.\right)\mathbb P\left(S_b\leq 1-\frac{1}{\mu_1 N^{\alpha}}\right)\\
&+\mathbb P\left(\mathcal B\left|S_b> 1-\frac{1}{\mu_1 N^{\alpha}}\right.\right)\mathbb P\left(S_b> 1-\frac{1}{\mu_1 N^{\alpha}}\right)\\
\leq&\mathbb P\left(\mathcal B\left|S_b\leq 1-\frac{1}{\mu_1 N^{\alpha}}\right.\right)+\mathbb P\left(S_b> 1-\frac{1}{\mu_1 N^{\alpha}}\right)\label{pod-block-1}\\
\leq&\left(1-\frac{1}{\mu_1 N^{\alpha}}\right)^{\mu_1 N^\alpha \log N}+\mathbb P\left(\sum_{i=1}^b S_i>b-\frac{b}{\mu_1 N^{\alpha}}\right)\label{pod-block-2}\\
\leq& \frac{1}{N} + \frac{\mathbb E\left[\max\left\{\sum_{i=1}^b S_i-\lambda - \frac{k\log N}{\sqrt{N}}, 0\right\}\right]}{b-\lambda -\frac{k\log N}{\sqrt{N}}-\frac{b}{\mu_1 N^{\alpha}}}. \label{pod-block-3}\\
\leq& \frac{1}{N} + \frac{8\mu_{\max}}{b-\lambda} \frac{1}{\sqrt{N}\log N}. \label{pod-block-4}
\end{align} 
\eqref{pod-block-1} to \eqref{pod-block-2} holds because it denotes the probability of the event all sampled $d$ servers have $b$ jobs; \eqref{pod-block-2} to \eqref{pod-block-3} holds because $(1-\frac{1}{x})^x\leq \frac{1}{e}$ for $x\geq 1$ and the Markov inequality; \eqref{pod-block-3} to \eqref{pod-block-4} holds because of Theorem \ref{Thm:main} and $b-\lambda \geq \frac{8k\log N}{\sqrt{N}}+\frac{8b}{\mu_1N^\alpha}.$
The remaining analysis is the same.

Finally, for JIQ and I1F, we have not been able to bound $\mathbb P(\mathcal B).$ However,
\begin{align}
\mathbb P(\mathcal W)=&\mathbb P\left(S_1=1\right)\leq \mathbb P\left(\sum_{i=1}^bS_i\geq 1\right) \\
\leq& \mathbb P\left(\max\left\{\sum_{i=1}^bS_i-\lambda -\frac{k\log N}{\sqrt{N}}, 0\right\}\geq \frac{1}{N^{\alpha}}-\frac{k\log N}{\sqrt{N}}\right) \label{JIQ-zero-1}\\
\leq& \frac{\mathbb E \left[ \max\left\{\sum_{i=1}^bS_i-\lambda -\frac{k\log N}{\sqrt{N}}, 0\right\} \right] }{\frac{1}{N^{\alpha}}-\frac{k\log N}{\sqrt{N}}}\label{JIQ-zero-2}\\
\leq&\frac{\mathbb E \left[ \max\left\{\sum_{i=1}^bS_i-\lambda -\frac{k\log N}{\sqrt{N}}, 0\right\}\right]}{\frac{1}{2N^{\alpha}}} \label{JIQ-zero-3}\\
\leq& \frac{14\mu_{\max}}{N^{0.5-\alpha} \log N} \label{JIQ-zero-4}.
\end{align}
\eqref{JIQ-zero-1}-\eqref{JIQ-zero-2} holds because of the Markov inequality; \eqref{JIQ-zero-2}-\eqref{JIQ-zero-3} holds because $2k \leq \frac{N^{0.5-\alpha}}{\log N};$
\eqref{JIQ-zero-3}-\eqref{JIQ-zero-4} holds because of Theorem \ref{Thm:main}.
Given the choice of $k =\left(1+\frac{w_u b}{w_l}\right)\left(\frac{1+\mu_1+\mu_2}{w_l}+2\mu_1\right)$ in Theorem \ref{Thm:main}, we need the buffer size $b$ to be at the same order, which leads to the finite-buffer assumption.

\end{document}